\documentclass[11pt]{amsart}
\usepackage[utf8]{inputenc}

\usepackage{pdfsync}
\usepackage{amsthm}
\usepackage{amsmath}
\usepackage{amssymb}
\usepackage{mathtools}
\usepackage{aligned-overset}

\usepackage{tikz}
\usetikzlibrary{calc}
\usetikzlibrary{shapes}
\usetikzlibrary{snakes}

\usepackage{verbatim}
\usepackage[margin=1in]{geometry}

\usepackage[shortlabels]{enumitem}
\setlist{nosep}

\usepackage[font={small}]{caption}
\numberwithin{equation}{section}

\usepackage{thmtools}
\usepackage{thm-restate}

\newtheorem{theorem}{Theorem}[section]
\newtheorem{lemma}[theorem]{Lemma}

\newtheorem{proposition}[theorem]{Proposition}
\newtheorem{corollary}[theorem]{Corollary}

\newtheorem{obs}[theorem]{Observation}

\usepackage{etoolbox}
\newtheoremstyle{claimstyle}{5pt}{5pt}{\em}{}{\bf}{.}{5pt}{}
\theoremstyle{claimstyle}
\newtheorem{claim}{Claim}

\newenvironment{claimproof}{\noindent \textit{Proof of claim:}}{\hfill\ensuremath{\blacklozenge}\vspace{5pt}}
\AtEndEnvironment{proof}{\setcounter{claim}{0}}

\newtheoremstyle{stepstyle}{5pt}{5pt}{\em}{}{\bf}{.}{5pt}{}
\theoremstyle{stepstyle}
\newtheorem{pf-step}{Step}
\AtEndEnvironment{proof}{\setcounter{pf-step}{0}}

\theoremstyle{definition}

\newtheorem{question}[theorem]{Question}

\theoremstyle{remark}

\newcommand{\cH}{\ensuremath{\mathcal H}}

\newcommand{\cW}{\ensuremath{\mathcal W}}
\newcommand{\ordering}{\ensuremath{\preceq}}
\newcommand{\fwdnbr}{\ensuremath{N^\ordering}}
\newcommand{\fwddeg}{\ensuremath{d^\ordering}}

\newcommand{\eps}{\ensuremath{\varepsilon}}

\DeclareMathOperator{\rank}{r}
\DeclareMathOperator{\antirank}{ar}
\DeclareMathOperator{\vol}{vol}

\usepackage{graphicx}
\usepackage{xparse} 
\usepackage{geometry}
\usepackage{comment}
\usepackage{enumitem}
\usepackage{fancyhdr}
\usepackage[unicode,
bookmarksopen=true,
bookmarksopenlevel=1,
colorlinks=true,
linkcolor=darkblue,
linktoc=page,
citecolor=darkblue,
hypertexnames=false,   
]{hyperref}
\usepackage{tikz}

\definecolor{darkblue}{rgb}{0,0,0.5}
\definecolor{cerisepink}{rgb}{0.93, 0.23, 0.51}
\definecolor{chocolate(traditional)}{rgb}{0.48, 0.25, 0.0}
\definecolor{chromeyellow}{rgb}{1.0, 0.65, 0.0}

\newcommand{\mylabel}[2]{#2\def\@currentlabel{#2}\label{#1}}

\def\COMMENT#1{}

\usepackage[foot]{amsaddr}

\def\abhi#1{}
\def\Dongyeap#1{}
\def\tom#1{}

\let\abhi=\abhiOpt 
\let\Dongyeap=\DongyeapOpt 
\let\tom=\TomOpt 

\begin{document}
\title[Solution to a problem of Erd\H{o}s]{Solution to a problem of Erd\H{o}s on the chromatic index of hypergraphs with bounded codegree}
\date{\today}

\author[Kang]{Dong Yeap Kang}
\author[Kelly]{Tom Kelly}
\author[K\"uhn]{Daniela K\"uhn}
\author[Methuku]{Abhishek Methuku}
\author[Osthus]{Deryk Osthus}

\thanks{This project has received partial funding from the European Research
Council (ERC) under the European Union's Horizon 2020 research and innovation programme (grant agreement no. 786198, D.~K\"uhn and D.~Osthus).
The research leading to these results was also partially supported by the EPSRC, grant nos. EP/N019504/1 (D.~Kang, T. Kelly and D.~K\"uhn) and EP/S00100X/1 (A.~Methuku and D.~Osthus).
D.~Kang was supported by Institute for Basic Science (IBS-R029-Y6).
T.~Kelly was partially supported by the National Science Foundation under Grant No. DMS-2247078. A.~Methuku was partially supported by the SNSF grant 200021\_196965.}

\begin{abstract}
In 1977, Erd\H{o}s asked the following question: for any integers $t,n \in \mathbb{N}$, if $G_1 , \dots , G_n$ are complete graphs such that each $G_i$ has at most $n$ vertices and every pair of them shares at most $t$ vertices, what is the largest possible chromatic number of the union $\bigcup_{i=1}^{n} G_i$?  
The equivalent dual formulation of this question asks for the largest chromatic index of an $n$-vertex hypergraph with maximum degree at most $n$ and maximum codegree at most $t$.
For the case $t = 1$, Erd\H os, Faber, and Lov\' asz famously conjectured that the answer is $n$, which was recently proved by the authors for all sufficiently large $n$.
In this paper, we answer this question of Erd\H{o}s for $t \geq 2$ in a strong sense, by proving that every $n$-vertex hypergraph with maximum degree at most $(1-o(1))tn$ and maximum codegree at most $t$ has chromatic index at most $tn$ for any $t,n \in \mathbb{N}$. 
Moreover, equality holds if and only if the hypergraph is a $t$-fold projective plane of order $k$, where $n = k^2 + k + 1$.
Thus, for every $t \in \mathbb N$, this bound is best possible for infinitely many integers $n$. This result also holds for the list chromatic index.  
\end{abstract}
\maketitle 

\section{Introduction}\label{sec:intro}
A central theme in Combinatorics is how the local structure of a graph or hypergraph influences global parameters such as the chromatic number. A fundamental early result in this direction is the probabilistic proof by Erd\H{o}s~\cite{erdos1959} of the existence of graphs of high girth and high chromatic number. In particular, the fact that ``locally'' the graph has chromatic number two does not imply that globally the chromatic number is bounded. On the other hand, a famous conjecture where the local structure does completely determine the chromatic number was formulated by Erd\H{o}s, Faber, and Lov\'asz in 1972: if $G$ is the ``nearly disjoint" union of $n$ complete graphs of order $n$, then $G$ has chromatic number $n$. This was recently proved by the authors~\cite{KKKMO2021} for large $n$. In this paper, we build on some of the ideas of ~\cite{KKKMO2021} to provide an answer (Theorem~\ref{main-thm-erdos}) to a related question of Erd\H{o}s (Question~\ref{ques:erdos}) about colouring a union of complete graphs whose pairwise intersection is bounded.  
Our main result (Theorem~\ref{main-thm}) is a more general result that we prove in the ``dual setting'' of edge-colouring hypergraphs with bounded codegree.

\subsection{Answering a question of Erd\H os}
A \textit{proper colouring} of a graph $G$ is an assignment of colours to its vertices such that adjacent vertices are assigned different colours, and the \textit{chromatic number} of $G$, denoted $\chi(G)$, is the fewest number of colours needed to properly colour $G$. 

Erd\H{o}s~\cite[Problem 9]{erdos1979} asked the following question (see also~\cite[Problem 95]{chung1997} and~\cite{erdos1995} for the $t=2$ case) that considers the chromatic number of graphs which are a union of complete graphs with bounded pairwise intersection.

\begin{question}[Erd\H{o}s~\cite{erdos1979}]\label{ques:erdos}
Let $n,t \in \mathbb{N}$. If $G_1 , \dots , G_n$ are complete graphs, each on at most $n$ vertices, such that $|V(G_i) \cap V(G_j)| \leq t$ for all distinct $i, j \in [n]$, what is the largest possible chromatic number of $\bigcup_{i=1}^{n}G_i$?
\end{question}

In 1972, Erd\H{o}s, Faber, and Lov\'asz famously conjectured that for $t=1$, the answer to Question~\ref{ques:erdos} is $n$, and this was recently proved for all sufficiently large $n$ by the authors~\cite{KKKMO2021}. 
Building on the ideas of~\cite{KKKMO2021}, we prove the following theorem (Theorem~\ref{main-thm-erdos}) that answers the question for all $2 \leq t < \sqrt{n}$\COMMENT{If $n = k^2 + k + 1$, then $k^2 < n < (k + 1)^2$, so $\lfloor \sqrt n \rfloor = k$.  Theorem~\ref{main-thm-erdos} applies in ths case for $t \leq k$.} and sufficiently large $n$. The range when $t$ is larger is already covered asymptotically by an observation of  Hor\'{a}k and Tuza~\cite{ht1990}, see Section~\ref{sec:results}. We emphasize that since Theorem~\ref{main-thm-erdos} requires that $t \ge 2$, it does not imply our results in \cite{KKKMO2021}.  

\begin{theorem}\label{main-thm-erdos}
    There exists $n_0 \in \mathbb{N}$ such that the following holds for all $n,t \in \mathbb{N}$ where $n \geq n_0$ and $t\geq 2$.
    If $G_1, \dots, G_n$ are complete graphs, each on at most $n$ vertices, such that $|V(G_i) \cap V(G_j)| \leq t$ for all distinct $i, j \in [n]$, then $\chi(\bigcup_{i=1}^n G_i) \leq tn$.  
    Moreover, for infinitely many $k \in \mathbb N$, if $n = k^2 + k + 1$ and $t \leq k$,
    then there exist such $G_1, \dots, G_n$ such that $\bigcup_{i=1}^n G_i$ has $tn$ vertices and is complete (and in particular has chromatic number $tn$).
\end{theorem}

\subsection{The dual setting}

We prove Theorem~\ref{main-thm-erdos} in a ``dual setting''.  To that end, we recall the following notions.

 A \emph{hypergraph} $\cH$ is a finite set equipped with a ground set $V(\cH)$ whose elements are called \textit{vertices} of $\cH$, and every element $e \in \cH$ is called an \emph{edge} of $\cH$ and is equipped with a non-empty set $V(e) \subseteq V(\cH)$. 
 For convenience, we will often consider $\cH$ as a multiset $\{ V(e) : e \in \cH \}$ of subsets of $V(\cH)$ by identifying $e$ with $V(e)$ for each $e \in \cH$, if it is clear from the context.
 We will think of any edge $e \in \cH$ as \emph{consisting} of all the vertices in $V(e)$.
We emphasize that we allow $V(e)=V(f)$ for edges $e \neq f$ in $\cH$, i.e.~we allow ``multiple edges" in our hypergraph. A \textit{proper edge-colouring} of a hypergraph $\cH$
is an assignment of colours to its edges such that any two of its edges $e \neq f$ assigned the same colour satisfy $V(e) \cap V(f)= \varnothing$, and the \textit{chromatic index} of $\cH$, denoted $\chi'(\cH)$, is the fewest number of colours needed to properly edge-colour $\cH$. The \textit{dual} of a hypergraph $\cH$ is the hypergraph $\cH^*$ with vertex set $\cH$ and edge set $V(\cH)$ such that each edge $u \in \cH^* = V(\cH)$ is equipped with $V(u) \coloneqq \{ e \in \cH : u \in V(e) \}$.
(Clearly, the dual of $\cH^*$ is isomorphic to $\cH$ itself.)
The \textit{line graph} of $\cH$, denoted $L(\cH)$, is the graph $G\coloneqq L(\cH)$ where $V(G)$ is the edge set of $\cH$, and $e, f \in V(G)$ are adjacent in $G$ if $V(e)\cap V(f) \neq \varnothing$.  
Note that the chromatic index of $\cH$ is the chromatic number of its line graph.
The union of complete graphs in Question~\ref{ques:erdos} can be represented as the line graph of a hypergraph, as follows.
  \begin{enumerate}[label=(\theequation)]
    \stepcounter{equation}
  \item\label{fact:duality1} If $G_1, \dots, G_n$ are complete graphs and $\cH$ is the dual of the hypergraph $\{ e_i : i \in [n] \}$ with $V(e_i) \coloneqq V(G_i)$ for each $i \in [n]$, then the line graph of $\cH$ is isomorphic to $\bigcup_{i=1}^n G_i$ and so $\chi'(\cH) = \chi\left(\bigcup_{i=1}^n G_i\right)$.  
  \end{enumerate}

The \textit{degree} of a vertex $v \in V(\cH)$ in $\cH$, denoted $d_\cH(v)$, is the number of edges $e\in\cH$ such that $V(e)$ contains $v$, and the \textit{codegree} of two distinct vertices $u, v \in V(\cH)$ in $\cH$, denoted $d_\cH(u, v)$, is the number of edges $e \in \cH$ such that $V(e)$ contains both $u$ and $v$.  The \textit{maximum degree} of $\cH$, denoted $\Delta(\cH)$, is the maximum degree of a vertex of $\cH$, and the \textit{maximum codegree} of $\cH$, denoted $\Delta_2(\cH)$, is the maximum codegree of two distinct vertices of $\cH$.  
Note the following:
  \begin{enumerate}[label=(\theequation)]
    \stepcounter{equation}
  \item\label{fact:duality2} If $G_1, \dots, G_n$ are complete graphs and $\cH$ is the dual of the hypergraph $\{ e_i : i \in [n] \}$ with $V(e_i) \coloneqq V(G_i)$ for each $i \in [n]$, then $\Delta(\cH) = \max\{|V(G_i)| : i \in [n]\}$ and $\Delta_2(\cH) = \max\{|V(G_i)\cap V(G_j)| : i\neq j \in [n]\}$.
  \end{enumerate}

By~\ref{fact:duality1} and~\ref{fact:duality2}, Question~\ref{ques:erdos} can be answered by determining the largest chromatic index of an $n$-vertex hypergraph with maximum degree at most $n$ and maximum codegree at most $t$. 
In fact, we can see from the following that both questions are equivalent.

\begin{enumerate}[label=(\theequation)]
    \stepcounter{equation}
  \item[\ref*{fact:duality1}']\label{fact:duality3}   
  If $\cH$ is a hypergraph and $G_v$ is the subgraph of $L(\cH)$ induced by $\{e \in \cH : V(e) \ni v\}$ for every $v \in V(\cH)$, then $\bigcup_{v \in V(\cH)} G_v = L(\cH)$ and so $\chi(\bigcup_{v \in V(\cH)}G_v) = \chi'(\cH)$.

  \item[\ref*{fact:duality2}'] If $\cH$ is a hypergraph and $G_v$ is the subgraph of $L(\cH)$ induced by $\{e \in \cH : V(e) \ni v\}$for every $v \in V(\cH)$, then $G_v$ is a complete graph satisfying $|V(G_v)| = d_{\cH}(v) \leq \Delta(\cH)$ for every $v\in V(\cH)$ and $|V(G_v) \cap V(G_v)| = d_{\cH}(u, v) \leq \Delta_2(\cH)$ for every two distinct $u, v \in V(\cH)$.
\end{enumerate}

By considering this dual formulation, it is straightforward to show the ``moreover'' part of Theorem~\ref{main-thm-erdos}, as follows.  
A hypergraph $\cH$ is \textit{linear} if 
it has maximum codegree at most one
and \emph{intersecting} if $V(e) \cap V(f) \ne \varnothing$ for every $e,f \in \cH$.   
For $k \in \mathbb N$, a projective plane of order $k$ is a linear and intersecting hypergraph $\cH$ with $k^2 + k + 1$ vertices such that every edge has \textit{size} $k + 1$ (where the size of an edge $e \in \cH$ is $|V(e)|$) and every vertex lies in $k+1$ edges.
A projective plane of order $k$ exists if $k$ is a power of a prime number, and it is a longstanding conjecture in finite geometry that the converse also holds (see~\cite[Conjecture C.2]{ball2015} and~\cite{BR1949}). 
For any hypergraph $\cH$ and $t \in \mathbb{N}$, a \emph{$t$-fold} $\cH$ is a hypergraph obtained from $\cH$ by replacing each edge with $t$ distinct copies of it.
Now let $t, k \in \mathbb N$ for which there exists a projective plane of order $k$, and let $\cH_{t,k}$ be a $t$-fold projective plane of order $k$.
Let $n \coloneqq k^2 + k + 1$, and let $L_{t,k}$ be the line graph of $\cH_{t, k}$.  Note that $L_{t,k} = \bigcup_{i=1}^n G_i$ where $G_i$ is a complete graph\COMMENT{whose vertices correspond to edges of $\cH_{t,k}$ containing a fixed vertex $v_i \in V(\cH_{t,k})$} for all $i \in [n]$, and $|V(G_i) \cap V(G_j)| = t$ for all distinct $i, j \in [n]$.
If $t \leq k$, then each $G_i$ has $t(k+1) < n$ vertices, and since $\cH_{t, k}$ is intersecting, $L_{t, k}$ is complete on $tn$ vertices, as desired.

\subsection{New Results}
\label{sec:results}
Before mentioning our new results, we first recall list colouring.
For any graph $G$, a \textit{list assignment} for $V(G)$ is a collection of ``lists of colours'' $C(v)$ for every $v \in V(G)$, and a \emph{$C$-colouring} of $G$ is a proper colouring $\phi$ of $G$ such that $\phi(v) \in C(v)$ for every $v \in V(G)$.  
We say $G$ is \emph{$C$-colourable} if $G$ has a $C$-colouring, and the \textit{list chromatic number} of $G$, denoted $\chi_\ell(G)$, is the minimum $t$ such that $G$ is $C$-colourable for any list assignment $C$ for $V(G)$ satisfying $|C(v)| \geq t$ for every $v \in V(G)$.
Similarly, for any hypergraph $\cH$ and any list assignment $C$ for $V(L(\cH)) = \cH$, a \emph{$C$-edge-colouring} of $\cH$ is a proper $C$-colouring of the line graph $L(\cH)$ of $\cH$. We say $\cH$ is \emph{$C$-edge-colourable} if the line graph $L(\cH)$ of $\cH$ has a $C$-colouring, and the \textit{list chromatic index} of $\cH$, denoted $\chi'_\ell(\cH)$, is the list chromatic number $\chi_\ell(L(\cH))$ of $L(\cH)$.

In the dual setting, Theorem~\ref{main-thm-erdos} implies the following for $t \geq 2$ and $n$ sufficiently large: If $\cH$ is an $n$-vertex hypergraph with $\Delta(\cH) \leq n$ and $\Delta_2(\cH) \leq t$, then $\chi'(\cH) \leq tn$.
Our main result strengthens this result in three ways.  Firstly, we allow the maximum degree of $\cH$ to be at most $(1 - \eps)tn$ for any $\eps > 0$.  Secondly, we prove that the result actually holds more generally for list colouring.   
Thirdly, we characterize the hypergraphs for which equality holds in the bound; in particular, $\chi'(\cH) = tn$ holds only if $\cH$ is a $t$-fold projective plane.  The following is our main result.

\begin{restatable}{theorem}{mainthm}\label{main-thm}
For every $\eps>0$, there exists $n_0 \in \mathbb{N}$ such that the following holds for all $n,t \in \mathbb{N}$ where $n \geq n_0$.
If $\cH$ is an $n$-vertex hypergraph with $\Delta(\cH) \leq (1 - \eps)tn$ and  $\Delta_2(\cH) \leq t$, then $\chi_\ell'(\cH) \leq tn$. 
Moreover, $\chi_\ell'(\cH) = tn$ if and only if $\cH$ is a $t$-fold projective plane of order $k \in \mathbb{N}$, where $n = k^2 + k + 1$.
\end{restatable}

We remark that Kahn's proof~\cite{kahn1992coloring} of an asymptotic version of the Erd\H{o}s-Faber-Lov\'asz Conjecture only shows that $\chi'_\ell(\cH) \leq (1 + o(1))tn$ in Theorem~\ref{main-thm}.

By~\ref{fact:duality1} and \ref{fact:duality2}, Theorem~\ref{main-thm} provides the following strengthening of Theorem~\ref{main-thm-erdos}.

\begin{corollary}\label{cor:main-thm-dual}
For every $\eps>0$, there exists $n_0 \in \mathbb{N}$ such that the following holds for all $n,t \in \mathbb{N}$ where $n \geq n_0$.
If $G_1, \dots, G_n$ are complete graphs, each on at most $(1 - \eps)tn$ vertices, such that $|V(G_i) \cap V(G_j)| \leq t$ for all distinct $i, j \in [n]$, then $\chi_\ell(\bigcup_{i=1}^n G_i) \leq tn$.  Moreover, $\chi_\ell(\bigcup_{i=1}^n G_i) = tn$ if and only if the dual of the hypergraph $\{ e_i : i \in [n] \}$ with $V(e_i) \coloneqq V(G_i)$ for each $i \in [n]$ is a $t$-fold projective plane of order $k \in \mathbb{N}$, where $n = k^2 + k + 1$.
\end{corollary}

Note that the bound on the chromatic number in Theorem~\ref{main-thm-erdos} follows immediately from Corollary~\ref{cor:main-thm-dual} (with $\eps = 1/2$, say).

Note that for every $t \in \mathbb{N}$, there are infinitely many $n$ so that the upper bound in Theorem~\ref{main-thm} is tight. If $t > \sqrt{n}$ this does not imply that Theorem~\ref{main-thm-erdos} is best possible, since the extremal examples for Theorem~\ref{main-thm} in this case have maximum degree greater than $n$ (and thus the corresponding
complete graphs in Theorem~\ref{main-thm-erdos} would have more than $n$ vertices).
However, Hor\'{a}k and Tuza~\cite{ht1990} proved an asymptotically optimal bound in this case, by showing that every $n$-vertex hypergraph with maximum degree at most $n$ has chromatic index at most $n^{3/2}$ (where for $n = k^2+k+1$, a $k$-fold projective plane of order $k$ shows that this bound is asymptotically optimal\COMMENT{If $n = k^2+k+1$, the degree of every vertex of a $k$-fold projective plane of order $k$ is $k(k + 1) = n - 1$, so the size of each clique in the dual construction is at most $n$. Moreover, its chromatic index is $kn = (1+o(1))n^{3/2}$ as it is intersecting and there are $kn$ edges in total.}).

By~\ref{fact:duality1} and~\ref{fact:duality2}, the Erd\H os-Faber-Lov\'asz Conjecture is equivalent to the following: Every $n$-vertex linear hypergraph $\cH$ with maximum degree at most $n$ has chromatic index at most $n$.  
Although Theorem~\ref{main-thm} also applies when $t = 1$ (namely, when $\cH$ is linear), it does not imply the result of~\cite{KKKMO2021} because of its stronger maximum degree assumption.  However, it is still an interesting open problem to characterize the extremal examples for the Erd\H os-Faber-Lov\'asz Conjecture, and Theorem~\ref{main-thm} provides such a characterisation of the extremal examples with maximum degree at most $(1 - \eps)n$ for any $\eps > 0$ (for sufficiently large $n$).
Moreover, it is also open whether the Erd\H os-Faber-Lov\'asz Conjecture holds for list coloring, which Theorem~\ref{main-thm} confirms under this maximum degree assumption.
See~\cite{KKKMO2021_survey} for a more detailed discussion of these open problems.\COMMENT{For this problem, it is sufficient to determine which $n$-vertex linear hypergraphs with no size-one edges have chromatic index $n$, since if $\cH'$ is the hypergraph obtained from a linear hypergraph $\cH$ by removing all edges of size one, then $\chi'(\cH) = \max\{\chi'(\cH'), \Delta(\cH)\}$\COMMENT{since at most $\max\{0, \Delta(\cH) - k\}$ additional colours are needed to extend a proper edge-colouring of $\cH'$ that uses $k$ colours to a proper edge-colouring of $\cH$} and $\Delta(\cH') < |V(\cH)|$.  There are at least two additional families of such extremal examples.
First, every $n$-vertex graph with more than $(n - 1)^2 / 2$ edges has chromatic index at least $n$ when $n$ is odd, so every linear hypergraph obtained by adding edges to such a graph is also extremal.  Second, every hypergraph $\cH$ obtained from an $n$-vertex complete graph by replacing a complete subgraph with a single edge $e$ also has chromatic index $n$ if $n - |V(e)|$ is odd (and is linear); in fact, if $|V(e)| = n - 1$ then $\cH$ is intersecting, and we say $\cH$ is a \textit{near-pencil}.  These two families, together with the projective planes, may include all of the extremal examples.}

We also prove the following stability result which generalizes~\cite[Theorem 1.2]{KKKMO2021} and which is used in the proof of Theorem~\ref{main-thm}.
\begin{restatable}{theorem}{mainthmsecond}\label{main-thm-2}
    For every $\delta > 0$, there exist $n_0, \mu > 0$ such that the following holds for all $n, t \in \mathbb{N}$ where $n \geq n_0$.  
    If $\cH$ is an $n$-vertex hypergraph with 
    $\Delta(\cH) \leq (1 - \delta)tn$ and $\Delta_2(\cH) \leq t$
    such that the number of edges of size $(1 \pm \delta)\sqrt{n}$ in $\cH$ is at most $(1-\delta)tn$, then $\chi_\ell'(\cH) \leq (1-\mu)tn$.
\end{restatable}

Note that the condition on the number of edges of size close to $\sqrt{n}$ means that $\cH$ cannot be ``close" to a $t$-fold projective plane.\COMMENT{We remark that one can easily transform the proofs of Theorems~\ref{main-thm} and~\ref{main-thm-2} into randomised polynomial-time algorithms that with high probability give a $C$-edge-colouring for any list assignment $C$ for $\cH$ where every edge $e \in \cH$ satisfies $|C(e)| \ge tn$, or $|C(e)| \ge (1 - \mu)tn$ respectively.}

To show that $t$-fold projective planes are the only extremal examples for Theorem~\ref{main-thm}, we in fact show that if equality holds in Theorem~\ref{main-thm} for some $\cH$, then $\cH$ has $tn$ pairwise intersecting edges.
To deduce that $\cH$ is a $t$-fold projective plane in this case, we prove the following theorem. Theorem~\ref{thm:char_intersect} is a version of~\cite[Theorem 8]{furedi1986} (which applies to linear hypergraphs and generalizes the de Bruijn-Erd\H{o}s theorem~\cite{bruijn_erdos1948}) for hypergraphs with maximum codegree at most $t$. 
For any hypergraph $\cH$, let $e(\cH)$ be the number of edges in $\cH$. For any $v \in V(\cH)$, let $N[v] \coloneqq \bigcup_{e \in \cH : v \in V(e)} V(e)$. A \emph{near-pencil} on $n$ vertices is a linear hypergraph isomorphic to the hypergraph $\cH$ with $V(\cH) = \{ v , v_1 , \dots , v_{n-1} \}$ and $\cH = \{ e , e_1 , \dots , e_{n-1} \}$ where $V(e) = \{ v_1 , \dots , v_{n-1} \}$ and $V(e_i) = \{ v , v_i \}$ for $1 \leq i \leq n-1$. (A near-pencil is sometimes also referred to as a degenerate plane.)

\begin{restatable}{theorem}{CharIntersect}\label{thm:char_intersect}
Let $n,t \in \mathbb{N}$.  If $\cH$ is an $n$-vertex intersecting hypergraph with
$\Delta_2(\cH) \leq t$
and no edge of size one, then $e(\cH) \leq t \cdot \max_{v \in V(\cH)}|N[v]|$.
Moreover, if equality holds, then there exists $v \in V(\cH)$ such that $e(\cH) = t \cdot |N[v]|$, $d(x) = 0$ for every $x \in V(\cH) \setminus N[v]$, and $\cH[N[v]]$ is either
\begin{itemize}
    \item a $t$-fold projective plane of order $k$ (so there exists $k \in \mathbb{N}$ with $|N[v]| = k^2 + k + 1$), or
    \item a $t$-fold near-pencil.
\end{itemize}
\end{restatable}

Related questions are also discussed in the recent survey by the authors \cite{KKKMO2021_survey}. Moreover, Kelly, K\"uhn and Osthus~\cite{KKO2021} investigate the list chromatic number of graphs which are the ``nearly disjoint'' union of graphs of maximum degree $D$. This turns out to be closely related to a (wide open) conjecture of Vu~\cite{Vu02} on colouring graphs of small codegree.

\subsection{Organisation of the paper} 
In Section~\ref{sec:notation}, we introduce some notation that we follow throughout the paper. 
In Section~\ref{sec:overview}, we present an overview of the proof of Theorem~\ref{main-thm}.
In Section~\ref{large-edge-section}, we prove Lemmas~\ref{lemma:good-ordering} and \ref{lemma:good-ordering-extremal} that allow us to colour ``large'' edges efficiently. In Section~\ref{fpp-extremal-section}, we prove Lemma~\ref{extremal-case-lemma} which is used to colour edges having size at least $(1-o(1))\sqrt{n}$. In Section~\ref{section:proof}, we put everything together and prove Theorems~\ref{main-thm} and~\ref{main-thm-2}. Finally, in Section~\ref{sec:char} we prove Theorem~\ref{thm:char_intersect} that allows us to characterize the hypergraphs for which equality holds in Theorem~\ref{main-thm}. (In particular, Theorem~\ref{thm:char_intersect} is used in the proof of Theorem~\ref{main-thm}.)

\section{Notation}
\label{sec:notation}
For every integer $n \geq 0$, let $[n] \coloneqq \{ m \in \mathbb{N} : 1 \leq m \leq n \}$. For $a,b,c,d \in \mathbb{R}$ and $b>0$, we write $c = a \pm b$ if $a-b \leq c \leq a+b$ and if $c>0$ we write $d = (a \pm b)c$ if $(a-b)c \leq d \leq (a+b)c$. 

We will use the ``$\ll$'' notation when stating our results. The constants in the hierarchies used to state our results have to be chosen from right to left. More precisely, if we claim that a result holds whenever $0 < a \ll b \le 1,$ then this means that there exists a non-decreasing function $f : (0, 1] \mapsto (0, 1]$ such that the result holds for all $0 < a, b \le 1$ with $a \le f(b)$. We will not calculate these functions explicitly. Hierarchies with more constants are defined in a similar way.

Let $\cH$ be a hypergraph. A vertex $v \in V(\cH)$ is said to be \emph{contained} in an edge $e \in \cH$, or an edge $e \in \cH$ is said to \emph{contain} a vertex $v \in V(\cH)$ if $v \in V(e)$. The edges $f_1 , \dots , f_k \in \cH$ are \emph{pairwise intersecting} if the sets $V(f_1) , \dots , V(f_k)$ are pairwise intersecting. 
The~\emph{rank} of $\cH$ is defined as $\rank(\cH) \coloneqq \max_{e \in \cH} |V(e)|$, and the~\emph{antirank} of $\cH$ is defined as $\antirank(\cH) \coloneqq \min_{e \in \cH} |V(e)|$.

For $S \subseteq \cH$ and $u, v \in V(\cH)$, let $d_S(u)$ be the number of edges in $S$ containing $u$, and the codegree $d_{S}(u,v)$ of two distinct vertices $u$ and $v$ is $|E_{S}(u,v)|$, where $E_{S}(u,v)$ denotes the set of edges in $S$ containing both $u$ and $v$. 

For every $v \in V(\cH)$, let $N_{\cH} [v] \coloneqq \bigcup_{e \in \cH : v \in V(e)} V(e)$ and let $N_{\cH}(v) \coloneqq N_{\cH}[v] \setminus \{ v \}$.  
For every edge $e \in \cH$, the set $N_{\cH}(e)$ denotes the set of edges $f \in \cH \setminus \{e \}$ such that $V(e) \cap V(f) \ne \varnothing$. 
The subscript $\cH$ in $N_{\cH}[v]$, $N_{\cH}(v)$, and $N_{\cH}(e)$ may be omitted if it is clear from the context.  If $\cH'\subseteq \cH$ and $e\in \cH$, then we denote $N_\cH(e)\cap \cH'$ by $N_{\cH'}(e)$, even if $e\notin\cH'$.

If $\ordering$ is a linear ordering of the edges of a hypergraph $\cH$, for each $e\in\cH$, we define $\fwdnbr_\cH(e)\coloneqq \{f \in N_\cH(e) : f\ordering e \}$ and $\fwddeg_\cH(e)\coloneqq |\fwdnbr_\cH(e)|$.  We omit the subscript $\cH$ when it is clear from the context. 
For each $e\in \cH$, we also let $\cH^{\ordering e} \coloneqq \{f\in \cH : f \ordering e\}$.
For disjoint $\cH_1 , \dots , \cH_k \subseteq \cH$, we say $\cH_1 \ordering \dots \ordering \cH_k$ if $e_1 \ordering \dots \ordering e_k$ for all $e_i \in \cH_i$ and $i \in [k]$.

For every subset $S \subseteq \cH$,  the \emph{volume} of $S$ is defined as $\vol_{\cH}(S) \coloneqq \sum_{e \in S} \binom{|V(e)|}{2} / \binom{n}{2}$, where $n$ is the number of vertices of $\cH$. This measures the average codegree of $S$ if we regard $S$ as a spanning subhypergraph of $\cH$. In particular, if the maximum codegree of $\cH$ is at most $t$, then $\vol_{\cH}(S) \leq t$.

For every subset $S \subseteq V(\cH)$, let $\cH[S]$ be the hypergraph such that $V(\cH[S]) \coloneqq S$ and $\cH[S] \coloneqq \{ e \in \cH : V(e) \subseteq S \}$. If $V(e) \cap S \ne \varnothing$ for every $e \in \cH$, then let $\cH|_{S}$ be the hypergraph with $V(\cH|_{S}) \coloneqq S$ and $\cH|_{S} \coloneqq \{ e^* : e \in \cH \}$, where $V(e^*) \coloneqq V(e) \cap S$ for every $e \in \cH$.

\section{Proof overview}\label{sec:overview}
In this section, we will sketch the main ideas for proving Theorem~\ref{main-thm}. For the sake of clarity, we will discuss bounding the chromatic index rather than the list chromatic index, but this part still captures the main ideas of the proof.

Let $0 < 1/n_0 \ll 1/r_0 \ll 1/r_1 \ll \delta \ll \eps < 1$, let $\cH$ be an $n$-vertex hypergraph with $n \geq n_0$, $\Delta(\cH) \leq (1 - \eps)tn$, and $\Delta_2(\cH) \leq t$. 
We partition the edges of $\cH$ as follows. Let $\cH = \cH_{\rm lrg} \cup \cH_{\rm med} \cup \cH_{\rm sml}$, where
\begin{align*}
    \cH_{\rm lrg} \coloneqq \{ e \in \cH : |V(e)| > r_0 \}, \: \cH_{\rm med} \coloneqq \{ e \in \cH : |V(e)| \in (r_1 , r_0] \}, \: \cH_{\rm sml} \coloneqq \{ e \in \cH : |V(e)| \leq r_1 \}.
\end{align*}

Let us start by recalling the ideas of Kahn's~\cite{kahn1992coloring} proof that the Erd\H{o}s-Faber-Lov\'asz Conjecture holds asymptotically, which can also be used to show that $\chi'(\cH) \leq (1 + o(1))tn$.
Let $\ordering$ be an ordering of the edges of $\cH_{\rm lrg}$ such that for any $e, f \in \cH_{\rm lrg}$, we have $e \ordering f$ whenever $|V(e)| > |V(f)|$. Since $\Delta_2(\cH) \leq t$, for each $e \in \cH_{\rm lrg}$, it follows that $d^{\ordering}(e) + 1 \leq r_0 tn / (r_0 - 1)$ by Observation~\ref{obs:degree}. 
Thus, 
\begin{equation}\label{eq:large-greedy}
    \chi'(\cH_{\rm lrg}) \leq \frac{r_0 tn}{r_0 - 1}. 
\end{equation}
Note that since $\Delta_2(\cH) \leq t$, we have
$\Delta(\cH_{\rm med}) \leq tn / (r_1 - 1)$ by Observation~\ref{obs:degree}. Hence, the following theorem of Kahn~\cite{kahn1996asymptotically} shows that 
\begin{equation}\label{eq:medium-chromatic-bound}\chi'(\cH_{\rm med}) \leq \frac{2tn}{r_1 - 1}.
\end{equation}

\begin{restatable}[Kahn~\cite{kahn1996asymptotically}]{theorem}{KahnTheorem}\label{thm:kahn}
For any $k, \eps > 0$, there exists $\delta > 0$ such that the following holds. If $\cH$ is a hypergraph with $\rank(\cH) \leq k$, $\Delta(\cH) \leq D$, and $\Delta_2(\cH) \leq \delta D$, then $\chi_\ell'(\cH) \leq (1+\eps)D$.
\end{restatable}

Theorem~\ref{thm:kahn} can also be used to show that
\begin{equation}\label{eq:coloring-small-and-large-together}
    \chi'(\cH_{\rm sml} \cup \cH_{\rm lrg}) \leq \max \{ \chi'(\cH_{\rm lrg}),\: \lfloor (1 - \eps/2) tn \rfloor \}.
\end{equation}
Indeed, let $\phi_{\rm lrg} : \cH_{\rm lrg} \to [ \chi'(\cH_{\rm lrg})  ]$ be a proper edge-colouring of $\cH_{\rm lrg}$.
For each $e \in \cH_{\rm sml}$, 
let $C'(e) \coloneqq \{ \phi_{\rm lrg}(f) : V(e) \cap V(f) \ne \varnothing,\:f \in \cH_{\rm lrg} \}$. 
Since $\Delta_2(\cH) \leq t$ and every edge of $\cH_{\rm sml}$ has size at most $r_1$, every $e \in \cH_{\rm sml}$ intersects at most $r_1 t n / (r_0 - 1) \leq \eps tn / 10$ edges in $\cH_{\rm lrg}$ (by Observation~\ref{obs:degree}), so $|C'(e)| \leq \eps tn / 10$. Let $L(e) \coloneqq [\lfloor (1 - \eps/2) tn \rfloor] \setminus C'(e)$ for each $e \in \cH_{\rm sml}$. Then $|L(e)| \geq (1 - \eps) tn + \eps tn / 100$. 
As $\Delta(\cH_{\rm sml}) \leq \Delta(\cH) \le (1 - \eps)tn$, by Theorem~\ref{thm:kahn}, there exists a list $L$-edge-colouring $\phi_{\rm sml}$ of $\cH_{\rm sml}$, and by the definition of $L$, we can combine $\phi_{\rm lrg}$ and $\phi_{\rm sml}$ to obtain a proper edge-colouring of  $\cH_{\rm lrg} \cup \cH_{\rm sml}$ using at most $\max \{ \chi'(\cH_{\rm lrg}),\: \lfloor (1 - \eps/2) tn \rfloor \}$ colours, as desired.

By simply colouring $\cH_{\rm med}$ and $\cH_{\rm sml}\cup\cH_{\rm lrg}$ with disjoint sets of colours, we have
\begin{equation}
    \label{eq:coloring-medium-separately}
    \chi'(\cH) \leq \chi'(\cH_{\rm med}) + \chi'(\cH_{\rm{sml}} \cup \cH_{\rm{lrg}}).
\end{equation}
Combining \eqref{eq:large-greedy}, \eqref{eq:medium-chromatic-bound}, \eqref{eq:coloring-small-and-large-together} and \eqref{eq:coloring-medium-separately}, we have 
\begin{align*}\label{eqn:kahn_bound}
    \chi'(\cH) \leq \chi'(\cH_{\rm med}) + \chi'(\cH_{\rm sml} \cup \cH_{\rm lrg}) \leq \frac{2tn}{r_1 - 1} + \max \{ \chi'(\cH_{\rm lrg}) ,\: (1 - \eps/2)tn \} \leq (1 + \eps) tn.    
\end{align*}

In view of the above bound, the main challenge then lies in improving this bound to $tn$. Indeed, to prove that $\chi'(\cH) \leq tn$, we cannot use~\eqref{eq:large-greedy}, and we need to show that the edges of $\cH_{\rm lrg}$ can be properly coloured with $tn$ colours. Furthermore, since $\cH_{\rm lrg}$ could use all $tn$ colours (e.g., when $\cH_{\rm lrg}$ is a $t$-fold projective plane), 
we may not be able to let $\cH_{\rm lrg}$ and $\cH_{\rm med}$ use disjoint sets of colours as in~\eqref{eq:coloring-medium-separately}.

To overcome these difficulties, for any $\tau \in (0,1)$ with $1/r_1 \ll \tau < 1$, 
one of our key lemmas (Lemma~\ref{reordering-lemma}, which we call the ``Reordering Lemma'') finds an ordering $\ordering$ of the edges of  $\cH_{\rm med} \cup \cH_{\rm lrg}$ such that either
\begin{itemize}
    \item
    $d^{\ordering}(e) \leq (1 - \tau)tn$ for all $e \in \cH_{\rm med} \cup \cH_{\rm lrg}$, or
    
    \item
    there exists a subset $\cW$ of consecutive edges with `volume' $(1 - o_\tau (1))t$, and $d^\ordering (e) \leq (1 - \tau)tn$ for all $e \in \cH_{\rm med} \cup \cH_{\rm lrg}$ with $\cW \ordering e$.
\end{itemize}

We may assume that the second outcome holds, as otherwise~\eqref{eq:medium-chromatic-bound}--\eqref{eq:coloring-medium-separately} shows that $\chi'(\cH) \leq tn$. 
One of the key definitions is the~\emph{volume} $\vol_\cH(S)$ of a subset $S \subseteq \cH$, denoted $\vol_\cH(S) \coloneqq \sum_{e \in S} \binom{|V(e)|}{2}/\binom{n}{2}$. Note that $\vol_\cH(\cH) \leq t$ as $\Delta_2(\cH) \leq t$. 
We often make use of the following property that the volume is additive over disjoint subsets. 
\begin{obs}\label{obs:volume-additive}
For any disjoint $S_1 , S_2 \subseteq \cH$, we have $\vol_\cH(S_1) + \vol_\cH(S_2) = \vol_\cH(S_1 \cup S_2) \leq t$.
\end{obs}

Thus, applying the Reordering Lemma (Lemma~\ref{reordering-lemma}) again to the subhypergraph $\cH' \subseteq \cH$ consisting of the edges $e \ordering \cW$ with a suitable choice of $\tau$ (say $\tau'$), we can find an ordering $\ordering'$ of the edges of $\cH'$ so that $d^{\ordering'}(e) \leq (1 - \tau')tn$ for $e \in \cH'$, since otherwise we obtain a subset $\cW' \subseteq \cH'$ disjoint from $\cW$ with $\vol_\cH(\cW) + \vol_\cH(\cW') > t$, contradicting the assumption that $\Delta_2(\cH) \leq t$. 

Based on this idea, Lemma~\ref{reordering-lemma} is applied via two lemmas (Lemmas~\ref{lemma:good-ordering} and~\ref{lemma:good-ordering-extremal})  depending on the structure of $\cH$, and in particular, whether or not the number of edges of size $(1 \pm \delta) \sqrt{n}$ is less than $(1 - \delta) tn$.

If $\cH$ has fewer than $(1 - \delta) tn$ edges of size $(1 \pm \delta) \sqrt{n}$, then for any $\mu \in (0,1)$ with $1/r_1 \ll \mu \ll \delta < 1$, Lemma~\ref{lemma:good-ordering} provides an ordering $\ordering$ of the edges of $\cH_{\rm lrg}$ and a partition of $\cH_{\rm lrg}$ into three spanning subhypergraphs $\cH_2$, $\cW$, and $\cH_1$ satisfying~\ref{W-max-edge-size}--\ref{H2-edge-size}~ and~\ref{reordering-goodness1}--\ref{partition-order}. 
We will then show that the line graph of $\cW$ is `locally sparse' (in Lemma~\ref{local-sparsity-lemma}) and thus, by a result of Molloy and Reed \cite{MR02} (Theorem~\ref{local-sparsity-theorem}), we have that $\chi'(\cW) \leq (1 - \delta/500) tn$ (as in Corollary~\ref{sparsity-corollary}). Using this fact, we deduce that $\chi'(\cH_{\rm lrg}) \leq (1 - \mu) tn$ (in Lemma~\ref{lemma:large-extremal-col}). As before, we can then extend the colouring to $\cH_{\rm sml}$ (as in \eqref{eq:coloring-small-and-large-together}) and show that $\chi'(\cH) \leq tn$ by \eqref{eq:medium-chromatic-bound} and \eqref{eq:coloring-medium-separately}.

Otherwise, if $\cH$ has at least $(1 - \delta)tn$ edges of size $(1 \pm \delta) \sqrt{n}$, then the set of these edges has volume $(1 - O(\delta))t$. 
Using this fact and Observation~\ref{obs:volume-additive}, and by applying the Reordering Lemma (Lemma~\ref{reordering-lemma}) twice, we prove 
Lemma~\ref{lemma:good-ordering-extremal} which provides an ordering $\ordering$ and a partition of $\cH_{\rm med} \cup \cH_{\rm lrg}$ into three spanning subhypergraphs $\cH_3$, $\cH_2$, and $\cH_1$ satisfying \ref{H1-edge-size}, \ref{W-min-edge-size}, and \ref{reordering-goodness3}--\ref{partition-order2}. 
In this case, we cannot use disjoint sets of colours for $\cH_{\rm med}$ and $\cH_{\rm lrg}$, so we show that there is a proper edge-colouring of $\cH_{\rm med} \cup \cH_{\rm lrg}$ using at most $tn$ colours in which at most $\eps tn / 2$ colours are used for colouring the edges of $\cH_{\rm med}$. This allows us to extend the colouring to $\cH_{\rm sml}$ and show that $\chi'(\cH) \leq tn$ by the same argument used to prove \eqref{eq:coloring-small-and-large-together}. Let us finish by saying a few words about how we colour the edges of $\cH_{\rm med} \cup \cH_{\rm lrg}$. Here we will have $\antirank(\cH_3) \geq (1 - 2\delta)\sqrt{n}$, so we can properly colour its edges with $tn$ colours by Lemma~\ref{extremal-case-lemma}. Then using the fact that $\fwddeg_\cH(e) \leq tn - 2$ for every $e\in \cH_2$, we extend the colouring of $\cH_3$ to $\cH_2$ using no additional colours. Finally, since $d^{\ordering}(e) \leq \gamma tn$ for every $e \in \cH_1$ (for $\gamma \ll \eps$) and $\cH_{\rm med} \subseteq \cH_1$, we can properly colour the edges of $\cH_{1}$ using only a small subset of the colours, as desired.

\section{Reordering large edges}\label{large-edge-section}

In the proofs of Theorems~\ref{main-thm} and~\ref{main-thm-2}, we first colour the ``large'' edges (edges of size at least $r$ for some $r$ satisfying $1 / r \ll \eps$ in Theorem~\ref{main-thm} and $1 / r \ll \delta$ in Theorem~\ref{main-thm-2}) before extending the colouring to the remaining ``small'' edges.  This section is devoted to proving the following two lemmas, which provide additional structure among the ``large'' edges that enables us to colour them efficiently.

\begin{lemma}\label{lemma:good-ordering}
    Let $0 < 1/n_0 \ll 1/r \ll \mu \ll \delta \ll 1$, and let $n, t\in\mathbb N$ where $n \geq n_0$.  Let $\mathcal H$ be an $n$-vertex hypergraph
    with $\Delta_2(\cH) \leq t$ and $\antirank(\cH) > r$.
    If $\chi'_\ell(\cH) > (1 - \mu)tn$, then there exists a partition of $\cH$ into three spanning subhypergraphs $\cH_1$, $\cW$, and $\cH_2$ such that
  \begin{enumerate}[label=(P\arabic*)]
  \item\label{W-max-edge-size} 
  $\rank(\cW) \leq (1 + \delta)\antirank(\cW)$ and 
  \item\label{W-volume2} $\vol_\cH(\cW) \geq (1 - \delta)t$,
  \item\label{H2-edge-size} $\antirank(\cH_2) \geq \rank(\cW)$.
  \end{enumerate}
  and a linear ordering $\ordering$ of the edges of $\cH$ such that
  \begin{enumerate}[label=(FD\arabic*)]
  \item\label{reordering-goodness1} $\fwddeg_\cH(e) \leq (1 - 2\mu)tn$ for all $e\in \cH_1$, 
  \item\label{reordering-goodness2} $\fwddeg_\cH(e) \leq tn/2000$ for all $e\in \cH_2$, and
  \item\label{partition-order} 
  $\cH_2 \ordering \cW \ordering \cH_1$.
  \end{enumerate}
\end{lemma}

We remark that the coefficient 1/2000 from~\ref{reordering-goodness2} can be replaced with any sufficiently small constant $\delta^* \in (0,1)$ with $\delta^*/\delta$ sufficiently large. 

\begin{lemma}\label{lemma:good-ordering-extremal}
    Let $0 < 1/n_0 \ll 1/r_0 \ll  1/r_1 \ll \delta \ll \gamma \ll 1$, and let $n, t\in\mathbb N$ where $n \geq n_0$.  Let $\mathcal H$ be an $n$-vertex hypergraph 
    with $\Delta_2(\cH) \leq t$ and $\antirank(\cH) > r_1$.
    If the number of edges in $\cH$ of size $(1 \pm \delta)\sqrt{n}$ is at least $(1-\delta)tn$, then there exists a partition of $\cH$ into three spanning subhypergraphs $\cH_1$, $\cH_2$, and $\cH_3$ such that
  \begin{enumerate}[label=(P'\arabic*)]
  \item\label{H1-edge-size} 
  every $e \in \cH$ satisfying $|V(e)| \leq r_0$ is in $\cH_1$ and
  \item\label{W-min-edge-size} 
  $\antirank(\cH_3) \geq (1 - 2\delta) \sqrt{n}$,
  \end{enumerate}
  and a linear ordering $\ordering$ of the edges of $\cH$ such that
  \begin{enumerate}[label=(FD'\arabic*)]
  \item\label{reordering-goodness3} $\fwddeg_\cH(e) \leq tn - 2$ for every $e\in \cH_2$,
  \item\label{reordering-goodness4} $\fwddeg_\cH(e) \leq \gamma tn$ for every $e\in \cH_1$, and
  \item\label{partition-order2} 
  $\cH_3 \ordering \cH_2 \ordering \cH_1$.
  \end{enumerate}
\end{lemma}

In Section~\ref{section:proof}, we use Lemma~\ref{lemma:good-ordering} in the proof of Theorem~\ref{main-thm-2}.  Property~\ref{W-max-edge-size} enables us to apply a result of Molloy and Reed~\cite[Theorem~10.5]{MR02} (see Theorem~\ref{local-sparsity-lemma}) on colouring ``locally sparse'' graphs to efficiently colour the edges of $\cW$ from their lists, and properties~\ref{W-volume2} and~\ref{H2-edge-size}, together with~\ref{reordering-goodness2}, enable us to colour the edges of $\cH_2$ with a disjoint set of colours.  Then, we can complete the colouring (of the ``large'' edges) greedily by~\ref{reordering-goodness1} and~\ref{partition-order}.

Having proved Theorem~\ref{main-thm-2}, we may assume in the proof of Theorem~\ref{main-thm} that there are at least $(1 - \delta)tn$ edges in $\cH$ of size $(1 \pm \delta)\sqrt n$ for $\delta \ll \eps$, in which case we apply Lemma~\ref{lemma:good-ordering-extremal} and colour the hypergraph $\cH_3$ obtained from Lemma~\ref{lemma:good-ordering-extremal} first.  The argument for colouring $\cH_3$ is quite delicate, and the next section, Section~\ref{fpp-extremal-section}, is entirely devoted to it (see Lemma~\ref{extremal-case-lemma}).  We use~\ref{reordering-goodness3},~\ref{reordering-goodness4}, and~\ref{partition-order2} to extend this colouring greedily to the  remaining ``large'' edges (i.e.~edges of size at least $r_1$), while~\ref{reordering-goodness4} and~\ref{H1-edge-size} ensure that only a small proportion of colours are used on edges of ``medium'' size (i.e.~large edges of size at most $r_0$).
As mentioned earlier, the arguments build on ideas introduced in~\cite{KKKMO2021}.

\subsection{Reordering edges}

Our main tool in the proofs of Lemmas~\ref{lemma:good-ordering} and~\ref{lemma:good-ordering-extremal} is the following lemma, which generalises~\cite[Lemma 6.2]{KKKMO2021}.

\begin{lemma}\label{reordering-lemma}
  Let $0 < 1/n_0 \ll 1/r_1 \ll \tau, 1/K$ where $\tau < 1$, $K \geq 1$, and $1 - \tau - 7\tau^{1/4} / K > 0$, and let $t,n \in \mathbb{N}$ with $n \geq n_0$.
  If $\cH$ is an $n$-vertex hypergraph with 
  $\Delta_2(\cH) \leq t$ and $\antirank(\cH) > r_1$,
  then there exists a linear ordering $\ordering$ of the edges of $\cH$ such that at least one of the following holds.
  \begin{enumerate}[(\ref*{reordering-lemma}:a), topsep = 6pt]
  \item Every $e\in\cH$ satisfies $\fwddeg(e) \leq t(1 - \tau)n$.\label{reordering-good}
  \item There is a set $\cW\subseteq \cH$ such that
    \begin{enumerate}[(W1)]
    \item\label{W-max-size} 
    $\rank(\cW) \leq (1 + 3\tau^{1/4} K^3) \antirank(\cW)$ and
    \item\label{W-volume} $\vol_{\cH}(\cW) \geq t\frac{(1 - \tau - 7\tau^{1/4}/K)^2}{1 + 3\tau^{1/4} K^3}$.
    \end{enumerate}
    Moreover, if $e^*$ is the last edge of $\cW$, then
    \begin{enumerate}[(O1)]
    \item for all $f\in \cH$ such that $e^*\ordering f$ and $f\neq e^*$, we have $\fwddeg(f) \leq t(1 - \tau)n$ and\label{ordering-goodness}
    \item for all $e,f\in \cH$ such that $f\ordering e \ordering e^*$, we have $|V(f)| \geq |V(e)|$\label{ordering-by-size}.
    \end{enumerate}
    \label{reordering-volume}
  \end{enumerate}
    
\end{lemma}

Note that if (\ref{reordering-lemma}:a) of Lemma~\ref{reordering-lemma} holds, then one can list-colour  $\cH$ greedily with $t(1-\tau)n+1$ colours.
To prove Lemma~\ref{reordering-lemma}, we proceed roughly as follows: If we cannot find an ordering which satisfies (\ref{reordering-lemma}:a), then we consider a suitable ordering satisfying~\ref{ordering-goodness} and~\ref{ordering-by-size} for some $e^* \in \cH$ (a size-monotone ordering with the largest edge coming first is one example of an ordering satisfying \ref{ordering-goodness} and \ref{ordering-by-size} with the last edge $e^*$). 
Let $\ordering$ be one of such orderings that minimises $e(\cH^{\ordering e^*})$. 
If $\ordering$ violates (\ref{reordering-lemma}:a), one can then show that a set $\cW$ of edges which precede $e^*$ and are only a little larger than $e^*$ satisfies~\ref{W-max-size} and~\ref{W-volume}. The former follows by definition of $\cW$ and the latter is proven via a double counting argument which shows $\rm (i)$ that $\cW$ contains essentially all edges $f$ which (in the line graph of $\cH$) lie in the first or second neighbourhood of $e^*$ and $\rm (ii)$ there are not too many other edges.

To formalize this argument, we first need several ingredients.

\begin{obs}\label{obs:degree}
Let $r \geq 2$, and let $\cH$ be a hypergraph with $\Delta_2(\cH) \leq t$.
Then every vertex is contained in at most $t(n-1)/(r-1)$ edges of size at least $r$ in $\cH$.
\end{obs}
\begin{proof}
Let $\cH_r \coloneqq \{ e \in \cH : |e| \ge r \}$. Then for every $v \in V(\cH)$,
\begin{equation*}
    (r-1)d_{\cH_r}(v) \leq \sum_{e \colon v \in e \in \cH_r}(|V(e)|-1) = \sum_{w \colon v \ne w} d_{\cH_r}(v,w) \leq t(n-1),
\end{equation*}
where for the last inequality we used that $\cH$ (and hence $\cH_r$) has maximum codegree at most $t$. 
Thus, $d_{\cH_r}(v) \leq t(n-1)/(r-1)$, as desired.
\end{proof}

We remark that the lemmas that assume large $\antirank(\cH)$ do not necessarily require $\Delta(\cH) \leq (1 - o(1))tn$ as in Theorem~\ref{main-thm}, since $\Delta(\cH) \leq t(n - 1)/ (\antirank(\cH)-1)$ by Observation~\ref{obs:degree}.

The following observation bounds the number of edges intersecting a given edge in a significant proportion of its vertices.

\begin{obs}\label{obs:intersect}
Let $t \in \mathbb{N}$ and let $\cH$ be a hypergraph with $\Delta_2(\cH) \leq t$.
For every $e \in \cH$ and $\alpha \in (0,1)$ such that $\alpha|V(e)| \geq 2$, there are at most $2t/\alpha^2$ edges $f \in \cH$ such that $|V(e) \cap V(f)| \geq \alpha |V(e)|$.
\end{obs}
\begin{proof}
Fix $e \in \cH$. Let $N_\alpha$ be the number of edges $f \in \cH$ such that $|V(e) \cap V(f)| \geq \alpha |V(e)|$. Since every pair in $e$ is covered by at most $t$ edges in $\cH$, we have 
\begin{equation*}
    t \binom{|V(e)|}{2} \geq \sum_{f \in \cH} \binom{|V(e) \cap V(f)|}{2} \geq N_\alpha \binom{\alpha|V(e)|}{2}.
\end{equation*}
 This easily implies the desired bound $N_\alpha \leq 2t/\alpha^2$, since $\alpha |V(e)| \geq 2$.\COMMENT{Thus,
\begin{equation*}
    N_\alpha \leq t \binom{|V(e)|}{2} \binom{\alpha|V(e)|}{2}^{-1} \leq \frac{t|V(e)|(|V(e)|-1)}{\alpha|V(e)|(\alpha|V(e)|-1)} \leq \frac{t |V(e)|^2}{\alpha |V(e)| \alpha|V(e)|/2} = \frac{2t}{\alpha^2},
\end{equation*}
as desired.}
\end{proof}

\begin{proposition}\label{fwdnbr-prop}
  Let $t,n \in \mathbb{N}$ and $\alpha_1, \alpha_2, \tau \geq 0$.  Let $\cH$ be an $n$-vertex hypergraph with 
  $\Delta_2(\cH) \leq t$ and $\antirank(\cH) \geq 2 (1 + \alpha_2)^2$.
  Let $e\in \cH$, let $r\coloneqq |V(e)|$, let $m_1 \coloneqq |\{f \in N(e) : |V(f)| \geq (1 + \alpha_1)r\}|$, and let $m_2 \coloneqq |\{f \in N(e) : (1 + \alpha_1)r > |V(f)| \geq r / (1 + \alpha_2)\}|$.  Then
  \begin{equation}\label{fwddeg-inequality}\tag{i}
    (1 + \alpha_1)m_1 + \frac{m_2}{1 + \alpha_2} \leq t n \left ( 1 + \frac{1 + \alpha_2}{r - 1 - \alpha_2} \right ).
  \end{equation}
  Moreover, if $m_1 + m_2 \geq t(1 - \tau)n$ and $\alpha_1 > 0$, then
  \begin{equation}\label{fwddeg-far-bound}\tag{ii}
    m_1 \leq \left(\tau + \frac{1 + \alpha_2 + \alpha_2 r}{r - 1 - \alpha_2}\right)\frac{tn}{\alpha_1}.
  \end{equation}
\end{proposition}
\begin{proof}
First, we show~\eqref{fwddeg-inequality}.
Let $N_1 \coloneqq \{ f \in N(e) : V(f) \subseteq V(e) ,\ |V(f)| \geq r/(1+\alpha_2) \}$ and $N_2 \coloneqq \{ f \in N(e) : V(f) \not\subseteq V(e) ,\ |V(f)| \geq r/(1+\alpha_2) \}$. \COMMENT{Since $|V(f) \cap V(e)| \geq |V(e)|/(1+\alpha_2)$ for every $f \in N_1$,}By Observation~\ref{obs:intersect}, $|N_1| \leq (1+\alpha_2)^2 2t$. Thus, since $r \geq 2(1+\alpha_2)^2$, and $|V(f)| \leq |V(e)|=r$ for all $f \in N_1$, we have
\begin{align}\label{eqn:n1}
    \sum_{f \in N_1} (|V(f)|-1) \leq (r-1) (1+\alpha_2)^2 2t \leq tr^2.
\end{align}
Also note that 
\begin{align}\label{eqn:n2}
 \sum_{f \in N_2} (|V(f)|-1) \leq \sum_{f \in N(e)}|V(f) \cap V(e)||V(f) \setminus V(e)| = \sum_{u \in V(e),\:v \notin V(e)} d_{\cH}(u,v) \leq tr(n-r).
\end{align}
Combining~\eqref{eqn:n1} and~\eqref{eqn:n2}, we have
\begin{align}
\label{eq:newn3}
    \sum_{f \in N(e),|V(f)| \geq r/(1+\alpha_2)} (|V(f)|-1) \leq trn,
\end{align}
and thus
\begin{equation*}
    \sum_{f \in N(e) : |V(f)| \geq (1 + \alpha_1)r} \left((1 + \alpha_1)r - 1\right) + \sum_{f \in N(e) : (1 + \alpha_1)r > |V(f)| \geq r/(1 + \alpha_2)} \left( \frac{r}{1+\alpha_2} - 1\right) \leq trn.
\end{equation*}
Dividing both sides of this inequality by $r$ and rearranging terms, we obtain
\begin{equation}\label{fwddeg-inequality1}
    (1 + \alpha_1)m_1 + \frac{m_2}{1 + \alpha_2} \leq tn + \frac{m_1 + m_2}{r}.
\end{equation}
By \eqref{eq:newn3}, we also have
\begin{equation}\label{fwddeg-inequality2}
    m_1 + m_2 \leq \frac{trn}{r/(1 + \alpha_2) - 1} = \frac{(1 + \alpha_2)trn}{r - 1 - \alpha_2}.
\end{equation}
Hence, we obtain~\eqref{fwddeg-inequality} by replacing $m_1 + m_2$ with the bound in~\eqref{fwddeg-inequality2} in the right side of~\eqref{fwddeg-inequality1}.

Now we will show~\eqref{fwddeg-far-bound}. 
Combining the inequality $m_1 + m_2 \geq t(1 - \tau)n$  with~\eqref{fwddeg-inequality}\COMMENT{The left side of~\eqref{fwddeg-inequality} is $m_1 + m_2 + \alpha_1 m_1 + (1/(1 + \alpha_2) - 1)m_2$, so if $m_1 + m_2 \geq t(1 - \tau)n$, then $t(1 - \tau)n + \alpha_1 m_1 + (1/(1 + \alpha_2) - 1)m_2$ is at most the right side of~\eqref{fwddeg-inequality}.}, we obtain\COMMENT{From \eqref{fwddeg-inequality2} below we only used $m_2 \le \frac{(1 + \alpha_2)trn}{r - 1 - \alpha_2}$, so $(1-\frac{1}{1+\alpha_2})m_2 \leq \alpha_2  \frac{t r n}{r - 1 -\alpha_2}$} 
\begin{align*}
    \alpha_1 m_1  \leq t\tau n + \frac{tn(1+\alpha_2)}{r-1-\alpha_2} + \left (1 - \frac{1}{1 + \alpha_2} \right )  m_2
    & \overset{\eqref{fwddeg-inequality2}}{\leq} t\tau n + \frac{tn(1+\alpha_2)}{r-1-\alpha_2} + \alpha_2 \frac{trn}{r - 1 - \alpha_2}\\
    & = tn \left ( \tau + \frac{1+\alpha_2 + \alpha_2 r}{r-1-\alpha_2} \right ),
\end{align*}
as desired.
\end{proof}

Now we are ready to prove Lemma~\ref{reordering-lemma}.

\begin{proof}[Proof of Lemma~\ref{reordering-lemma}]
  Let $\alpha$ satisfy $1/r_1 \ll \alpha \ll \tau, 1/K$.
  Let $\ordering$ be an ordering of the edges in $\cH$ satisfying both~\ref{ordering-goodness} and~\ref{ordering-by-size} for some $e^*\in\cH$ such that $e(\cH^{\ordering e^*})$ is minimum. Such an ordering exists since every size-monotone ordering $\ordering$ of the edges in $\cH$, where $f \ordering e$ if $|V(f)| > |V(e)|$, satisfies~\ref{ordering-goodness} and~\ref{ordering-by-size} for $e^*$, where $e^*$ is the last edge in the ordering.
  
  We may assume that $e(\cH^{\ordering e^*}) > 1$, otherwise $\ordering$ satisfies~\ref{reordering-good}, as desired. First we prove the following claim.
  \begin{claim}\label{claim:numorder}
    For every $e \in \cH^{\ordering e^*}$, 
    \begin{equation*}
        |N(e) \cap \cH^{\ordering e^*}| > (1 - \tau) tn.
    \end{equation*}
    In particular, $\fwddeg(e^*) > (1 - \tau)tn$.
  \end{claim}
  \begin{claimproof}
  If $\fwddeg(e^*) = |N(e^*) \cap \cH^{\ordering e^*}| \leq (1 - \tau)tn$, then the predecessor of $e^*$ also satisfies~\ref{ordering-goodness} and~\ref{ordering-by-size}, contradicting the choice of $e(\cH^{\ordering e^*})$ to be minimum. If $|N(e) \cap \cH^{\ordering e^*}| \leq (1 - \tau) tn$ for some $e \in \cH^{\ordering e^*} \setminus \{ e^* \}$, then we can move $e$ to be the successor of $e^*$ in the ordering $\ordering$, also contradicting the choice of $e(\cH^{\ordering e^*})$ to be minimum.
  \end{claimproof}
  
  Let $r \coloneqq |V(e^*)|$, and let $\cW \coloneqq \{f \ordering e^* : |V(f)| < (1 + 3 \tau^{1/4} K^3)r\}$, so that~\ref{W-max-size} is satisfied. 
  Note that $|V(f)| \geq r = |V(e^*)|$ for any $f \in \cW$, since $e^*$ satisfies~\ref{ordering-by-size}. 
  
  For the rest of the proof, we aim to show~\ref{W-volume}, which will ensure that $\cW$ satisfies~\ref{reordering-volume}.
  
  To that end, let $X \coloneqq \{e \in \fwdnbr(e^*) : |V(e)| < (1 + K\sqrt\tau)r\}$ and $X' \subseteq X$ be the collection of edges $e \in X$ with $|V(e) \cap V(e^*)| \leq \alpha r$.

  In order to bound $|\cW|$ from below, we first prove the following bound on $|X'|$.
  \begin{claim}\label{nbrs-far-ahead-of-e^*-bound}
  We have
    \begin{equation*}
        |X'| \geq \left(1 - \tau - \frac{2\sqrt\tau}{K}\right)tn.
    \end{equation*}
  \end{claim}
    \begin{claimproof}
      By Observation~\ref{obs:intersect}, we have 
      \begin{equation*}
          |X'| \geq |X| - \frac{2t}{\alpha^2} \geq |X| - \frac{\sqrt{\tau}}{2K} tn
      \end{equation*}
      and by Claim~\ref{claim:numorder} (where $e^*$ plays the role of $e$), we have 
      \begin{equation*}
          |X| \geq \left (1 - \tau \right )tn - |\fwdnbr(e^*)\setminus X|.
      \end{equation*}
    Moreover, by~\ref{ordering-by-size}, edges in $\fwdnbr(e^*)$ have size at least $r$, so $\fwdnbr(e^*) \setminus X  = \{e \in \fwdnbr(e^*) : |V(e)| \ge (1 + K\sqrt\tau)r\}$. Thus, we may apply Proposition~\ref{fwdnbr-prop}\eqref{fwddeg-far-bound} with $K \sqrt \tau$ and $0$ playing the roles of $\alpha_1$ and $\alpha_2$, respectively,\COMMENT{and with $m_1 = |\fwdnbr(e^*)\setminus X|$ and $m_2 = |X|$} to obtain\COMMENT{we used $\tau \gg 1/r_1 \ge 1/r$ below}
    \begin{equation*}
        |\fwdnbr(e^*)\setminus X| \leq \left(\tau + \frac{1}{r-1}\right)\frac{tn}{K\sqrt \tau} \leq \frac{3 \sqrt{\tau}}{2K}tn.
    \end{equation*}
    By combining the three inequalities above, the claim follows.
    \end{claimproof}
    
 In the next claim, we provide a lower bound on the number of neighbours an edge of $X'$ has in $\cW \setminus N(e^*)$.
  \begin{claim}\label{X-nbrs-in-W-bound}
  For every $e\in X'$,
  \begin{equation*}
    |N(e)\cap (\cW\setminus N(e^*))| \geq \left (1 - \tau - \frac{7\tau^{1/4}}{K}\right)tn - (1 + K \sqrt\tau)tr^2.
  \end{equation*}
  \end{claim}
  \begin{claimproof}
  Fix $e\in X'$. Note that
  \begin{equation}\label{eqn:boundclaim3}
      |N(e) \cap (\cW \setminus N(e^*))| = |N(e) \cap \cH^{\ordering e^*}| - |(N(e) \cap \cH^{\ordering e^*}) \setminus \cW| - |N(e) \cap N(e^*) \cap \cW|.
  \end{equation}
  
  Since $|N(e) \cap \cH^{\ordering e^*}| \geq (1-\tau)tn$ by Claim~\ref{claim:numorder}, it suffices to bound the last two terms appropriately.
  
  First, we aim to bound $|(N(e) \cap \cH^{\ordering e^*}) \setminus \cW|$.
  For every $f \in \cH^{\ordering e^*} \setminus \cW$, by~\ref{ordering-by-size} and the definition of $\cW$, it follows that\COMMENT{here we used $\frac{(1 + 3 \tau^{1/4} K^3)}{(1 + K^2 \tau^{1/4})}r \ge (1+ K \sqrt{\tau})r$ which holds since $1 + 3 \tau^{1/4} K^3 \ge 1 + K^3 \tau^{3/4} + K^2 \tau^{1/4} + K \sqrt{\tau}$.} $|V(f)| \geq (1 + 3 K^3 \tau^{1/4})r \geq (1 + K^2 \tau^{1/4})|V(e)|$. 
  Now we apply Proposition~\ref{fwdnbr-prop}\eqref{fwddeg-far-bound} to $e$ with $K^2\tau^{1/4}$ and $K\sqrt\tau$ playing the roles of $\alpha_1$ and $\alpha_2$ respectively. Let $m_1$ and $m_2$ be defined as in Proposition~\ref{fwdnbr-prop}.  Then $|(N(e)\cap \cH^{\ordering e^*})\setminus \cW)| \leq m_1$ and $m_1 + m_2 \geq (1 - \tau)tn$ by Claim~\ref{claim:numorder}. Thus, we obtain\COMMENT{The proposition gives the following in the denominator below $|V(e)|-1-\alpha_2 = |V(e)|-1- K \sqrt{\tau}$ and we used $|V(e)|-1- K \sqrt{\tau} \ge |V(e)|/2$, i.e., $|V(e)| \ge 2(1 + K\sqrt{\tau})$ but this is true since we assumed $1/r_1 \ll 1/K$ and $|V(e)| \ge r_1$.} 
  \begin{equation}\label{nbrhood-outside-W-bound}
    |(N(e) \cap \cH^{\ordering e^*}) \setminus \cW| \leq \left(\tau + \frac{1 + K\sqrt\tau + K\sqrt\tau |V(e)|}{|V(e)|/2}\right)\frac{tn}{K^2\tau^{1/4}} \leq \frac{6\tau^{1/4}}{K}tn,
  \end{equation}
  where we used $|V(e)| \geq r_1$.\COMMENT{In the second inequality we used $\left(\tau + \frac{1 + K\sqrt\tau + K\sqrt\tau |V(e)|}{|V(e)|/2}\right)\frac{tn}{K^2\tau^{1/4}} \leq \left (\tau + \frac{2 K \sqrt{\tau} |V(e)|}{|V(e)|/2} \right) \frac{tn}{K^2\tau^{1/4}} \leq \left (\tau + 4 K \sqrt{\tau} \right) \frac{tn}{K^2\tau^{1/4}} \leq \frac{6 \tau^{1/4}}{K} tn$}
  
  Now we aim to bound $|N(e)\cap \fwdnbr(e^*) \cap \cW|$.
  Let $N_1 \coloneqq \{f\in \cW : V(e) \cap V(e^*) \cap V(f) \ne \varnothing \}$ and $N_2 \coloneqq (N(e)\cap \fwdnbr(e^*) \cap \cW) \setminus N_1$.
  Then   by~\ref{ordering-by-size}, Observation~\ref{obs:degree} and the fact that $|V(e) \cap V(e^*)| \leq \alpha r$, we have
  \begin{equation*}
      |N_1| \leq \sum_{v \in V(e) \cap V(e^*)} d_{\cW} (v) \leq \alpha r \cdot \frac{t (n-1)}{r - 1}  \leq 2 \alpha tn.
  \end{equation*}
  Also, $|N_2| \leq t|V(e) \setminus V(e^*)||V(e^*) \setminus V(e)| \leq t(1 + K\sqrt \tau)r^2$, since $V(f)$ must intersect both $V(e) \setminus V(e^*)$ and $V(e^*) \setminus V(e)$ for every $f \in N_2$, and $\cH$ has maximum codegree at most $t$. 
  Hence
  \begin{equation}\label{eqn:nbrhood-intersection}
      |N(e)\cap \fwdnbr(e^*) \cap \cW| = |N_1| + |N_2| \leq 2 \alpha t n + (1 + K\sqrt \tau)tr^2.
  \end{equation}
  
  Combining Claim~\ref{claim:numorder},~\eqref{eqn:boundclaim3},~\eqref{nbrhood-outside-W-bound}, and~\eqref{eqn:nbrhood-intersection}, we obtain\COMMENT{since $\alpha \ll \tau, 1/K$}
  \begin{align*}
      |N(e) \cap (\cW \setminus N(e^*))| &\geq (1-\tau)tn - \frac{6\tau^{1/4}}{K}tn - 2\alpha tn - (1 + K\sqrt \tau)tr^2\\
      &\geq \left(1 - \tau - \frac{7 \tau^{1/4}}{K} \right) tn - (1 + K\sqrt \tau)tr^2,
  \end{align*}
  as desired.
  \end{claimproof}
  
  For every $f \in \cW\setminus N(e^*)$, since $\cH$ has maximum codegree at most $t$, we have
  \begin{equation}\label{W-nbrs-in-X-bound}
    |N(f) \cap X'| \leq |N(f) \cap N(e^*)| \leq t|V(f)||V(e^*)| \leq t(1 + 3\tau^{1/4}K^3)r^2.
  \end{equation}
  
  Since
  \begin{align*}
    \sum_{e\in X'}|N(e)\cap(\cW\setminus N(e^*))| &= |\{(e, f) : e \in X',\ f \in \cW\setminus N(e^*),\ e \in N(f)\}|\\
    &= \sum_{f\in \cW\setminus N(e^*)} | N(f) \cap X'|,
  \end{align*} 
  by combining Claim~\ref{X-nbrs-in-W-bound} and~\eqref{W-nbrs-in-X-bound}, we have
  \begin{equation*}
    |\cW\setminus N(e^*)| \geq  |X'|\left(\frac{(1 - \tau - 7\tau^{1/4}/K)n}{(1 + 3\tau^{1/4} K^3)r^2} - \frac{1 + K\sqrt\tau}{1 + 3\tau^{1/4} K^3}\right),
  \end{equation*}
  and since $X' \subseteq N(e^*) \cap \cW$, this implies
\begin{align*}
  |\cW| &\geq |X'|\left(\frac{(1 - \tau - 7\tau^{1/4}/K)n}{(1 + 3\tau^{1/4} K^3)r^2} + 1 - \frac{1 + K\sqrt\tau}{1 + 3\tau^{1/4} K^3}\right) \geq |X'| \frac{(1 - \tau - 7\tau^{1/4}/K)n}{(1 + 3\tau^{1/4} K^3)r^2} \\
  &\overset{{\bf Claim}~\ref{nbrs-far-ahead-of-e^*-bound}}{\geq} \left(1 - \tau - \frac{2\sqrt\tau}{K}\right)\left(\frac{1 - \tau - 7\tau^{1/4}/K}{1 + 3\tau^{1/4} K^3}\right)\frac{tn^2}{r^2}.
\end{align*}
Therefore,\COMMENT{in the last inequality below we used $(1 - \frac{1}{r})(1-\tau-\frac{2 \sqrt{\tau}}{K}) \ge 1 - \tau - \frac{2 \sqrt{\tau}}{K} - \frac{(1-\tau-\frac{2 \sqrt{\tau}}{K})}{r} \ge 1 - \tau - \frac{2 \sqrt{\tau}}{K} -\frac{1}{r}$, which is at least $1 - \tau - \frac{7 \tau^{1/4}}{K}$ since $\frac{1}{r} \le  \frac{1}{r_1} \ll \tau, \frac{1}{K}$.}
\begin{equation*}
    \vol_{\cH}(\cW) \geq |\cW|\binom{r}{2} / \binom{n}{2} \geq t \frac{(1 - \tau - 7\tau^{1/4}/K)^2}{1 + 3\tau^{1/4} K^3},
\end{equation*}
so $\cW$ satisfies~\ref{W-volume}, as desired, completing the proof of Lemma~\ref{reordering-lemma}.
\end{proof}

\subsection{Proofs of Lemmas~\ref{lemma:good-ordering} and~\ref{lemma:good-ordering-extremal}}

Now we can prove Lemmas~\ref{lemma:good-ordering} and~\ref{lemma:good-ordering-extremal}.  In both proofs, we apply Lemma~\ref{reordering-lemma} twice and combine the resulting orderings to obtain the desired ordering $\ordering$ of $\cH$.

\begin{proof}[Proof of Lemma~\ref{lemma:good-ordering}]
    First, we apply Lemma~\ref{reordering-lemma} to $\cH$ with $2\mu$ and $1$ playing the roles of $\tau$ and $K$, respectively, to obtain an ordering $\ordering_1$.  
    If $\ordering_1$ satisfies~\ref{reordering-good}, then $\chi'_\ell(\cH) < (1 - \mu)tn$, so we assume~\ref{reordering-volume} holds.  Let $\cW$ be the set $\cW$ obtained from~\ref{reordering-volume}, let $e^*$ be the last edge of $\cW$ in $\ordering_1$, and let $\cH_1 \coloneqq \cH\setminus \cH^{\ordering_1 e^*}$.  Let $f^*$ be the edge of $\cW$ which comes first in $\ordering_1$, and let $\cH_2 \coloneqq \cH \setminus \{e \in \cH : f^*\ordering_1 e\}$. By the choices of $\tau$ and $K$, and since $\mu \ll \delta$, we have\COMMENT{since $\tau = 2 \mu$ and $K =1$ we have $1+ 3\tau^{1/4} K^3 = 1+ 3 \cdot 2^{1/4}\mu^{1/4} \le 1+4\mu^{1/4}$.} $\rank(\cW)\leq (1 + 4\mu^{1/4})|V(e^*)| \leq (1 + \delta)\antirank(\cW)$ and $\vol_\cH(\cW) \geq t(1 - \mu^{1/5})^3 \geq (1 - \delta)t$, 
    so $\cW$ satisfies~\ref{W-max-edge-size} and~\ref{W-volume2}, as desired, and by~\ref{ordering-by-size} of ~\ref{reordering-volume}, we may assume without loss of generality that every $e\in \cH$ satisfying $f^* \ordering_1 e \ordering_1 e^*$ is in $\cW$, so $\cH$ is partitioned into $\cH_1$, $\cW$, and $\cH_2$, as required, and $\cH_2$ satisfies~\ref{H2-edge-size}, as desired.
    
    Now we reapply Lemma~\ref{reordering-lemma} to $\cH_2$ and show that the resulting ordering satisfies~\ref{reordering-goodness1}--\ref{partition-order}, as follows.  Apply Lemma~\ref{reordering-lemma} with $\cH_2$, $1 - 1/2000$, and $2000^2$ playing the roles of $\cH$, $\tau$, and $K$, respectively, to obtain an ordering $\ordering_2$.  Since $\cW\cap \cH_2 = \varnothing$ and $\cH$ has maximum codegree at most $t$, we have $\vol_\cH(\cW) + \vol_\cH(\cH_2) \leq t$.  Thus, $\ordering_2$ satisfies~\ref{reordering-good}, because~\ref{reordering-volume} would imply there is a set $\cW'\subseteq \cH_2$ disjoint from $\cW$ with $\vol_{\cH}(\cW') > \delta t$, contradicting~\ref{W-volume2}.  
    Combine $\ordering_1$ and $\ordering_2$ to obtain an ordering $\ordering$ of $\cH$ where 
    \begin{itemize}
    \item if $f\in\cH_1 \cup \cW$, then $e\ordering f$ for every $e\in \cH^{\ordering_1 f}$, and
    \item if $f\in\cH_2$, then $e \ordering f$ for every $e\in \cH_2^{\ordering_2 f}$.
    \end{itemize}
    Since $\cH_1$ and $\ordering_1$ satisfy~\ref{ordering-goodness} of~\ref{reordering-volume} with $\tau = 2\mu$,~\ref{reordering-goodness1} holds, and since $\cH_2$ and $\ordering_2$ satisfy~\ref{reordering-good} with $\tau = 1 - 1/2000$,~\ref{reordering-goodness2} holds.  Finally, by the definition of $\ordering$,~\ref{partition-order} holds, as desired.
\end{proof}

\begin{proof}[Proof of Lemma~\ref{lemma:good-ordering-extremal}]
    Let $\mu$ satisfy $1 / r_1 \ll \mu \ll \delta$, and let $\cW \coloneqq \{e \in \cH : |V(e)| = (1 \pm \delta)\sqrt n\}$.  By assumption,
    \begin{equation}\label{eqn:volW}
        \vol_\cH(\cW) \geq \left.(1 - \delta)tn\binom{(1 - \delta)\sqrt n}{2}\middle / \binom{n}{2}\right. \geq (1 - 3\delta)t.
    \end{equation}
    First, apply Lemma~\ref{reordering-lemma} to $\cH$ with $1 - \gamma$ and $\gamma^{-2}$ playing the roles of $\tau$ and $K$, respectively, to obtain an ordering $\ordering_1$.  If $\ordering_1$ satisfies~\ref{reordering-good}, then the lemma holds with $\cH_3, \cH_2 \coloneqq \varnothing$ and $\cH_1 \coloneqq \cH$ and $\ordering_1$ playing the role of $\ordering$, so we may assume that~\ref{reordering-volume} holds.
    Let $\cW_1$ be the set $\cW$ obtained from~\ref{reordering-volume}, let $e^*_1$ be the last edge of $\cW_1$, let $r_2\coloneqq |V(e^*_1)|$, let $\cH_1 \coloneqq \cH\setminus \cH^{\ordering_1e^*_1}$, and let $\cH^{\mathrm{left}} \coloneqq \cH\setminus \cH_1$.
  By the choices of $\tau$ and $K$, and since $\delta \ll \gamma \ll 1$, we have
  \begin{enumerate}[label = {($\mathrm{W}_1\arabic*)$}]
  \item\label{W1-max-size} $\rank(\cW_1) \leq r_2 / \gamma^{10}$ and
  \item\label{W1-volume} $\vol_\cH(\cW_1) \geq t \gamma^{20} > 3\delta t$.
  \end{enumerate}
  
    By~\eqref{eqn:volW} and~\ref{W1-volume}, $\cW \cap \cW_1 \neq \varnothing$, so by~\ref{W1-max-size}, $r_2 \geq \gamma^{10}(1 - \delta)\sqrt n > r_0$.  Thus, by~\ref{ordering-by-size} of~\ref{reordering-volume}, 
    $\cH_1$ satisfies~\ref{H1-edge-size}, as required.
  
    Now apply Lemma~\ref{reordering-lemma} to $\cH^{\mathrm{left}}$ with $\mu$ and $1$ playing the roles of $\tau$ and $K$, respectively, to obtain an ordering $\ordering_2$, and combine $\ordering_1$ and $\ordering_2$ to obtain an ordering $\ordering$ of $\cH$ where 
    \begin{itemize}
    \item if $f\in\cH_1$, then $e\ordering f$ for every $e\in \cH^{\ordering_1 f}$, and
    \item if $f\in\cH^{\mathrm{left}}$, then $e \ordering f$ for every $e\in (\cH^{\mathrm{left}})^{\ordering_2 f}$.
    \end{itemize}
    If $\ordering_2$ satisfies~\ref{reordering-good}, then $\ordering$ satisfies~\ref{reordering-goodness3} with $\cH_2 \coloneqq \cH^{\mathrm{left}}$ and~\ref{reordering-goodness4}\COMMENT{by (O1) of~\ref{reordering-volume} since we set $\tau = 1 - \gamma$ when we defined $\ordering_1$}, and the lemma holds with $\cH_3 \coloneqq \varnothing$, so we may assume that~\ref{reordering-volume} holds.
    Let $\cW_2$ be the set $\cW$ obtained from~\ref{reordering-volume}, let $e^*_2$ be the last edge of $\cW_2$ in $\ordering_2$, let $r_3 \coloneqq |V(e^*_2)|$, let $\cH_2 \coloneqq \cH^{\mathrm{left}}\setminus (\cH^{\mathrm{left}})^{\ordering_2 e^*_2}$, and let $\cH_3 \coloneqq \cH^{\mathrm{left}}\setminus\cH_2$.  Now $\cH_1$, $\cH_2$, and $\cH_3$ indeed partition $\cH$, and $\ordering$ satisfies~\ref{partition-order2}, as required.  By the choices of $\tau$ and $K$, and since $\mu \ll \delta\ll 1$, we have
  \begin{enumerate}[label = {($\mathrm{W}_2\arabic*)$}]
  \item\label{W2-max-size} $\rank(\cW_2) \leq (1 + 3\mu^{1/4})r_3$ and
  \item\label{W2-volume} $\vol_\cH(\cW_2) \geq t(1 - \mu^{1/5})$. 
  \end{enumerate}
  
  By~\eqref{eqn:volW} and~\ref{W2-volume}, $\cW \cap \cW_2 \neq \varnothing$, so by~\ref{W2-max-size}, $r_3 \geq (1 - \delta)\sqrt n / (1 + 3\mu^{1/4}) \geq (1 - 2\delta)\sqrt n$.  Thus, by~\ref{ordering-by-size} of~\ref{reordering-volume}, $\cH_3$ satisfies~\ref{W-min-edge-size}, as required.  Finally, since $\cH_1$ and $\ordering_1$ satisfy~\ref{ordering-goodness} of~\ref{reordering-volume} with $\tau = 1 - \gamma$,~\ref{reordering-goodness4} holds, and since $\cH_2$ and $\ordering_2$ satisfy~\ref{ordering-goodness} of~\ref{reordering-volume} with $\tau = \mu$,~\ref{reordering-goodness3} holds, as desired.  
\end{proof}

\section{The extremal case}\label{fpp-extremal-section}

In this section, we prove Lemma~\ref{extremal-case-lemma} which proves the special case of Theorem~\ref{main-thm} when all edges have size at least $(1-o(1))\sqrt{n}$. 
In the proof of Theorem~\ref{main-thm} in Section~\ref{section:proof}, this lemma will be used to colour the edges of $\cH_3$ obtained from an application of Lemma~\ref{lemma:good-ordering-extremal}.

\begin{lemma}\label{extremal-case-lemma}
Let $0 < 1/n_0 \ll \delta \ll 1$, and let $n,t \in \mathbb{N}$ where $n \geq n_0$. If $\cH$ is an $n$-vertex hypergraph with $\Delta_2(\cH) \leq t$ and $\antirank(\cH) \geq (1 - \delta) \sqrt n$,
then $\chi_\ell'(\cH) \leq tn$. Moreover, if $\chi_\ell'(\cH) = tn$, then $\cH$ is intersecting.
\end{lemma}

Note that $\vol_\cH(\cH) \leq t$ and $\antirank(\cH) \geq (1 - \delta)\sqrt{n}$ imply that 
\begin{align*}
    \chi_\ell'(\cH) \leq |\cH| \leq t \binom{n}{2} \binom{(1 - \delta) \sqrt{n}}{2}^{-1} \leq (1 + O(\delta))tn,
\end{align*}
where each colour class has size at most one. 
Thus, to prove Lemma~\ref{extremal-case-lemma}, 
the general strategy is to find a large matching in the complement of the line graph of $\cH$.
Such a matching naturally gives rise to a proper edge-colouring of $\cH$ with colour classes of sizes one and two. More precisely, for a hypergraph $\cH$, recall that $L(\cH)$ denotes the line graph of $\cH$ and $\overline{L(\cH)}$ denotes its complement.
First we need the following observation. It directly follows from the lemma~\cite[p.~152]{erdos1979choosability} that $\chi_\ell (K_{2 \ast r}) = r$ for all $r \in \mathbb{N}$, where $K_{2 \ast r}$ is a complete $r$-partite graph with all parts of size two. It also follows immediately from~\cite[Lemma~4.2]{KP18}.
\COMMENT{They only prove it for complete multipartite graphs with parts of size two, but we can reduce to this case by first greedily coloring vertices in parts of size one.  It also follows immediately from~\cite[Lemma~4.2]{KP18}.}

\begin{obs}\label{lem:chromindex}
Let $k,n \in \mathbb{N}$, and let $\cH$ be an $n$-vertex hypergraph. If there is a matching $N$ in $\overline{L(\cH)}$ of size at least $e(\cH) - k$, then $\chi_\ell'(\cH) \leq k$.
\end{obs}

A pair $\{e, f\}\subseteq \cH$ of distinct edges in an $n$-vertex hypergraph $\cH$ is \textit{$t$-useful} if $V(e) \cap V(f) \ne \varnothing$ and $|N(e) \cap N(f)| \leq tn - 3$.

\begin{proposition}\label{lem:chromindex0}
Let $n,t \in \mathbb{N}$, let $\cH$ be an $n$-vertex hypergraph with more than $tn$ edges, and let $r \coloneqq e(\cH) - tn$.
If there exist distinct pairwise intersecting edges $e_1 , \dots , e_{2r+2} \in \cH$ such that $\{e_{2i-1} , e_{2i}\}$ is $t$-useful for every $i \in [r+1]$, then $\chi_\ell'(\cH) < tn$.
\end{proposition}
\begin{proof}
By Observation~\ref{lem:chromindex}, it suffices to show that there exists a matching $N$ in $\overline{L(\cH)}$ of size $r+1  = e(\cH)-tn + 1$.

Let $i \in [r+1]$. Suppose there are distinct edges $f_1 , \dots , f_{i-1} \in \cH \setminus \{e_1 , \dots , e_{2r+2}\}$, where $V(f_j)$ does not intersect at least one of $V(e_{2j-1})$ or $V(e_{2j})$ for every $j \in [i - 1]$.  Now we choose $f_i \in \cH \setminus \{e_1 , \dots , e_{2r+2} , f_1 , \dots , f_{i-1} \}$ such that $V(f_i)$ does not intersect at least one of $V(e_{2i-1})$ or $V(e_{2i})$, as follows. Let
\begin{equation*}
    S \coloneqq \cH \setminus ((N(e_{2i-1}) \cap N(e_{2i})) \cup \{ e_{2i-1}, e_{2i} , f_1 , \dots , f_{i-1} \}).
\end{equation*}

For each $f \in S$, the set $V(f)$ does not intersect at least one of $V(e_{2i-1})$ or $V(e_{2i})$. Since, by definition, $|N(e_{2i-1}) \cap N(e_{2i})| \leq tn-3$ and $i-1 \leq r = e(\cH)-tn$,
\begin{equation*}
    |S| \geq e(\cH) - (tn-3) - 2 - (i-1) \geq 1.
\end{equation*}
Thus, $S \neq \varnothing$, so there exists $f_i \in S$. Since $e_1 , \dots , e_{2r+2}$ are pairwise intersecting, we have $S \cap \{e_1 , \dots , e_{2r+2} \} = \varnothing$, so $f_i \in \cH \setminus \{e_1 , \dots , e_{2r+2}\}$, as desired.

Therefore, let $f_1, \dots, f_{r+1} \in \cH \setminus \{e_1 , \dots , e_{2r+2}\}$ be chosen using the above procedure.
For every $j \in [r+1]$, $V(f_j)$ does not intersect at least one of $V(e_{2j-1})$ or $V(e_{2j})$, so this yields a matching $N$ in $\overline{L(\cH)}$ of size $r+1$, completing the proof.
\end{proof}

The following proposition shows that if we find sufficiently many $t$-useful pairs (satisfying some additional properties) then $\chi'_\ell(\cH) < tn$. To find such $t$-useful pairs will use Propositions~\ref{na-prop0} and \ref{na-prop}.\COMMENT{If $|A|$ is small (at most $t/(2\delta)$, say), then by Prop~\ref{na-prop0}, we know that all distinct intersecting pairs in $A$ are $t$-useful, so we can apply the first bullet of Prop~\ref{lem:chromindex1}. If $|A|$ is large and $|A^-|$ is small (at most $t/(4\delta)$, say), then by Prop~\ref{na-prop}, we know that all distinct intersecting pair (with small intersection) in $A$ are $t$-useful, so we can apply the second bullet of Prop~\ref{lem:chromindex1}. If both $|A|$ and $|A^-|$ are large, and $|A^+|$ is smaller than $\alpha \sqrt{n}|A^-|$, then $A^-$ and $A^+ \cup B$ playing the roles of $A$ and $B$, by Prop~\ref{na-prop}, we know that all distinct intersecting pair (with small intersection) in $A^-$ are $t$-useful, so we can apply the second bullet of Prop~\ref{lem:chromindex1}. Thus we ensure that both $|A|$ and $|A^-|$ are large, and $|A^+|$ is at least $\alpha \sqrt{n}|A^-|$.}

\begin{proposition}\label{lem:chromindex1}
Let $\alpha \in (0,1)$, let $n,t \in \mathbb{N}$, let $\cH$ be an $n$-vertex hypergraph such that $e(\cH) > tn$, $\Delta_2(\cH) \leq t$, and $\antirank(\cH) \geq 2/\alpha$. 
Let $r \coloneqq e(\cH) - tn$.  If some $A \subseteq \cH$ satisfies either
\begin{enumerate}[(i)]
    \item\label{hypo:A-bound} $|A| \geq 6r$, and $\{ e,f \}$ is $t$-useful for every distinct intersecting $e,f \in A$, or
    
    \item\label{hypo:A-bound-intersect} $|A| \geq 6r + 2t/\alpha^2$, and $\{ e,f \}$ is $t$-useful for every distinct intersecting $e,f \in A$ with $|V(e) \cap V(f)| \leq \alpha \max(|V(e)|,|V(f)|)$,
\end{enumerate}
then $\chi'_\ell(\cH) < tn$.
\end{proposition}
\begin{proof}
Note that by assumption, $r \geq 1$.  
Let $A$ be any subset of $\cH$ satisfying either~\ref{hypo:A-bound} or~\ref{hypo:A-bound-intersect}.
Let $N$ be a matching of maximum size in $\overline{L(\cH)}$. If $|N| > r$, then the lemma follows from Observation~\ref{lem:chromindex}, so we may assume that $|N| \leq r$. 
By the maximality of $N$, all pairs $e,f \in A \setminus V(N)$ are adjacent in $L(\cH)$, thus $e$ and $f$ are intersecting in $\cH$.

Now we claim that we can choose $2r+2$ distinct (pairwise intersecting) edges $a_1 , \dots , a_{2r+2} \in A \setminus V(N)$ such that $\{ a_{2i-1} , a_{2i} \}$ is $t$-useful for each $i \in [r+1]$, so that we can apply Proposition~\ref{lem:chromindex0} to $\{a_1 , \dots , a_{2r + 2} \}$ to complete the proof.  Since $|V(N)| \leq 2r$, we have
\begin{equation}\label{eqn:a_minus_vn}
    |A \setminus V(N)| - 2r \geq |A| - 4r.
\end{equation}

First suppose~\ref{hypo:A-bound} holds.  Then $|A| \geq 6r$, so by~\eqref{eqn:a_minus_vn},
\begin{equation*}
    |A \setminus V(N)| - 2r \geq 2r \geq 2.
\end{equation*}
Thus, we can choose $2r+2$ distinct pairwise intersecting edges $a_1 , \dots , a_{2r+2} \in A \setminus V(N)$ which, by our assumption, satisfy that
$\{ a_{2i-1} , a_{2i} \}$ is $t$-useful for each $i \in [r+1]$,  as claimed.

Otherwise, if~\ref{hypo:A-bound-intersect} holds, then again by~\eqref{eqn:a_minus_vn}, 
\begin{equation}\label{eqn:a_minus_vn2}
    |A \setminus V(N)| - 2r \geq \frac{2t}{\alpha^2} + 2.
\end{equation}

Now we claim that we can choose $2r+2$ distinct $a_1 , \dots , a_{2r+2} \in A \setminus V(N)$ such that for each $i \in [r+1]$, we have $|V(a_{2i-1}) \cap V(a_{2i})| \leq \alpha |V(a_{2i-1})|$. Then by our assumption $\{ a_{2i-1} , a_{2i} \}$ is $t$-useful for each $i \in [r+1]$, allowing us to apply Proposition~\ref{lem:chromindex0} to $\{a_1 , \dots , a_{2r + 2} \}$ and complete the proof again.

Suppose for some $0 \leq i \leq r$, we have chosen $2i$ distinct $a_1 , \dots , a_{2i} \in A \setminus V(N)$ such that $|V(a_{2j-1}) \cap V(a_{2j})| \leq \alpha |V(a_{2j-1})|$ for $j \in [i]$. 
Since $|A \setminus (V(N) \cup\{ a_1 , \dots , a_{2i} \})| \geq  2 + 2t/\alpha^2$ by~\eqref{eqn:a_minus_vn2}, there exist $a_{2i+1} , a_{2i+2} \in A \setminus (V(N) \cup \{ a_1 , \dots , a_{2i} \})$ such that $|V(a_{2i+1}) \cap V(a_{2i+2})| \leq \alpha |V(a_{2i+1})|$ by Observation~\ref{obs:intersect}.  Thus, iterating this process, we can choose $2r + 2$ distinct $a_1, \dots, a_{2r + 2} \in A \setminus V(N)$ such that $|V(a_{2i-1}) \cap V(a_{2i})| \leq \alpha |V(a_{2i-1})|$ for each $i \in [r+1]$, as claimed.
\end{proof}

\begin{proposition}\label{na-prop0}
  Let $0 < 1/n_0 \ll \delta \ll 1$, and let $t,k,n \in \mathbb{N}$ where $n \geq n_0$ and $k^2 - k + 2 \leq n \leq k^2 + k + 1$. 
  Let $\cH$ be an $n$-vertex hypergraph with 
  $\Delta_2(\cH) \leq t$ and $\antirank(\cH) \geq (1 - \delta) \sqrt n$
  such that 
  $\cH$ has at most $t/(2\delta)$ edges of size at most $k$. If $e \in \cH$ satisfies $|V(e)| \leq k$, then $|N(e)| \leq tn - 3$.
\end{proposition}
\begin{proof}
Fix $e \in \cH$ such that $|V(e)| \leq k$. Let $m$ be the number of edges of size at most $k$ in $N(e)$. By our assumption, $m \leq t/(2\delta)$. Since $\cH$ has maximum codegree at most $t$, we have\COMMENT{$e$ is fixed. For each $f$ look at the complete bipartite graph with $|f|$ vertices in total and non-empty parts $f \setminus e$ and $f \cap e$. The smallest number of edges it can have is $|f|-1$. Adding up over all $f$, each pair $xy$ with $x \in e$ and $y \not \in e$ is counted at most $t$ times}
\begin{align*}
    \sum_{f \in N(e) : V(f) \not\subseteq V(e)} (|V(f)|-1) &\leq \sum_{f \in N(e)}|V(f)\setminus V(e)| |V(f) \cap V(e)| \\
    &= \sum_{u \in V(e) , v \notin V(e)} d(u,v)\\ 
    & \leq t|V(e)||V(\cH) \setminus V(e)| \leq tk(n-k).
\end{align*}
On the other hand, since every edge in $\cH$ has size at least $ (1-\delta)\sqrt{n} \geq (1-2\delta)k \geq (1-2\delta) |V(e)|$, by Observation~\ref{obs:intersect}, there are at most $2t / (1-2\delta)^2$ edges $f \in \cH$ with $V(f) \subseteq V(e)$. Thus,
\begin{align*}
    \sum_{f \in N(e) : V(f) \not\subseteq V(e)} (|V(f)|-1) &\geq \sum_{f \in N(e)} (|V(f)|-1) - |V(e)| \cdot \frac{2t}{(1-2\delta)^2}\\
    &\geq k(|N(e)| - m) + ((1-\delta)\sqrt{n} - 1)m - \frac{2kt}{(1-2\delta)^2}.
\end{align*}
Combining the lower and upper bound on $\sum_{f \in N(e) : V(f) \not\subseteq V(e)} (|V(f)|-1)$, we obtain
\begin{align*}
    |N(e)| \leq t\left (n-k + \frac{2}{(1-2\delta)^2} \right ) + m \left (1 - \frac{(1-\delta) \sqrt{n} - 1}{k} \right ) \leq t(n-4) + 2\delta m \leq tn-3,
\end{align*}
as desired.\COMMENT{for the second inequality we used $1 - \frac{(1-\delta) \sqrt{n} - 1}{k} \le 1 - \frac{(1- 2 \delta)k }{k}  = 2 \delta$ since $(1-\delta)\sqrt{n} \geq (1-2\delta)k+1$}
\end{proof}

\begin{proposition}\label{na-prop}
  Let $0 < 1/n_0 \ll \delta \ll \alpha \ll 1$, and let $t,k,n \in \mathbb{N}$ where $n \geq n_0$ and $k^2 - k + 2 \leq n \leq k^2 + k + 1$.
  
  Let $\cH$ be an $n$-vertex hypergraph with 
  $\Delta_2(\cH) \leq t$ and $\antirank(\cH) \geq (1 - \delta) \sqrt n$, 
  and let $e, f \in \cH$ be distinct intersecting edges of size at most $k$ such that $|V(e) \cap V(f)| \leq \alpha k$. If either
  \begin{enumerate}[(i)]
      \item\label{na-prop:(k-1)-edge} at least one of $e$ or $f$ has size at most $k-1$, 
      
      \item\label{na-prop:intersection-bound} $|V(e) \cap V(f)| \geq 2$, or
      
      \item\label{na-prop:k-1-bound} there exists a vertex in $V(e) \cap V(f)$ which is contained in at most $t / (4 \delta)$ edges of size at most $k-1$,
  \end{enumerate}
   then $\{e,f \}$ is $t$-useful.
\end{proposition}
\begin{proof}
Let $i \coloneqq |V(e) \cap V(f)| \in [\lfloor \alpha k \rfloor]$. Since $\cH$ has maximum codegree at most $t$, we have
\begin{equation}\label{eqn:common-nbr-hood-bound}
    |N(e) \cap N(f)| \leq t(|V(e)|-i)(|V(f)|-i) + \sum_{w \in V(e) \cap V(f)}(d(w) - 2)
\end{equation}
and by Observation~\ref{obs:degree},\COMMENT{in the last inequality below we used $i\frac{t(n-1)}{(1-\delta)\sqrt{n} - 1} \leq i\frac{t(k^2+k)}{(1-\delta)(k-1) - 1}$ which is at most $(1+2 \delta)k t i$ provided $\frac{(k+1)}{(1-\delta)(k-1) - 1} \leq 1+2 \delta $ but this holds since $1/k \ll \delta$}
\begin{equation}\label{eqn:common-degree-bound}
    \sum_{w \in V(e) \cap V(f)}(d(w) - 2) \leq i\frac{t(n-1)}{(1-\delta)\sqrt{n} - 1} \leq (1 + 2\delta) kt i.
\end{equation}
Let $f(x) \coloneqq (|V(e)|-x)(|V(f)|-x) + (1+2\delta)kx$ for $x \in [1, \alpha k]$.
Since $|V(e)|,|V(f)| \geq (1-\delta)\sqrt{n} \geq (1-2\delta)k$,
\begin{equation*}
    f'(x) = -(|V(e)| + |V(f)|) + 2x + (1 + 2\delta)k \leq -2(1 - 2\delta)k + 2\alpha k + (1 + 2 \delta)k < -k/2 < 0,
\end{equation*}
so $f(x)$ is maximised when $x$ is taken as small as possible. Note that $|N(e) \cap N(f)| \leq tf(i)$ by~\eqref{eqn:common-nbr-hood-bound} and~\eqref{eqn:common-degree-bound}.
    
If~\ref{na-prop:(k-1)-edge} holds, then $f(1) \leq (k - 1)(k - 2) + (1 + 2\delta)k \leq k^2 - 3k/2$, so $|N(e) \cap N(f)| \leq t f(i) \leq t f(1) \leq t(k^2 - 3k/2) \leq tn-3$.  Hence, $\{e,f\}$ is $t$-useful, as desired.

If~\ref{na-prop:intersection-bound} holds, then $i \geq 2$ and $f(2) \leq (k-2)^2 + (1+2\delta)2k \leq k^2 - 3k/2$, so $|N(e) \cap N(f)| \leq t f(i) \leq t f(2) \leq t(k^2 - 3k/2) \leq tn-3$, so again $\{e,f\}$ is $t$-useful.

Otherwise we may assume $i = 1$ and~\ref{na-prop:k-1-bound} holds. Let $V(e) \cap V(f) = \{ w \}$, and let $m$ be the number of edges of size at most $k-1$ containing $w$.  By~\ref{na-prop:k-1-bound}, $m \leq t/(4\delta)$. 
Here we need an estimate on $d(w)$ which is more precise than~\eqref{eqn:common-degree-bound}. 
By the definition of $m$ and the fact that $\cH$ has maximum codegree at most $t$, we have
\begin{equation*}
    m ((1-\delta) \sqrt{n} - 1) + (d(w) - m)(k-1) \leq \sum_{e'\: : \: w \in V(e')}(|V(e')|-1) = \sum_{v \in V(\cH) \setminus \{w \} } d_\cH(v,w) \leq t(n-1),
\end{equation*}
which implies
\begin{equation}\label{eqn:d_w}
    d(w) \leq \frac{t(n-1)}{k-1} + m \left ( 1 - \frac{(1-\delta)\sqrt{n} - 1}{k-1} \right )
    \leq \frac{t(n-1)}{k-1} + 2 \delta m \leq \frac{t(n-1)}{k-1} + \frac{t}{2}.
\end{equation}

Thus, by~\eqref{eqn:common-nbr-hood-bound} and~\eqref{eqn:d_w},
\begin{align*}
    |N(e) \cap N(f)| - tn & \leq t(|V(e)|-1)(|V(f)|-1) + d(w) - 2 - tn \\ 
    & \leq t\left((k-1)^2 + \frac{n-1}{k-1} + \frac{1}{2} - n \right) - 2 \\
    & = t\frac{(k-1)^3 - 1 - (k-2)n + (k-1)/2}{k-1} - 2 \\
    & < t\frac{(k-2)(k^2 - k + 1 - n + 3/4)}{k-1} - 2\\
    & < - 2,
\end{align*}
so $\{e, f\}$ is a $t$-useful pair, as desired.
\end{proof}

Now we prove Lemma~\ref{extremal-case-lemma}.

\begin{proof}[Proof of Lemma~\ref{extremal-case-lemma}]
Choose $\alpha > 0$ so that $0 < 1/n_0 \ll \delta \ll \alpha \ll 1$. 
If $e(\cH) < tn$, then $\chi_\ell'(\cH) < tn$ since every edge can be given a different colour. If $e(\cH) = tn$ and $\cH$ is not intersecting, then by Observation~\ref{lem:chromindex}, $\chi_\ell'(\cH) < tn$. Thus, we may assume that $e(\cH) > tn$. We will show that $\chi_\ell'(\cH) < tn$.  Let  $r\coloneqq e(\cH) - tn$, let $k$ be a positive integer such that
\begin{equation}\label{eqn:rangen}
   k^2 - k + 2 = (k-1)^2 + (k-1) + 2 \leq n \leq k^2 + k + 1.
\end{equation}

Now we divide $\cH$ into the following subsets. Let $A^- \coloneqq \{e \in \cH \: : \: |V(e)| \leq k-1 \}$, let $A^+ \coloneqq \{e \in \cH \: : \: |V(e)| = k \}$, let $A \coloneqq A^- \cup A^+$, and let $B \coloneqq \{e \in \cH \: : \: |V(e)| \geq k+1 \}$. 

By~\eqref{eqn:rangen} and since every edge in $\cH$ has size at least $(1-\delta)\sqrt{n}$, it follows that\COMMENT{To get the second inequality below, note that $\frac{k(k-1)}{n-1} \geq \frac{(k^2 + k + 1) - (2k+1)}{n} \geq \frac{n - 3\sqrt{n}}{n}$ by~\eqref{eqn:rangen}, since $k = (1+o(1))\sqrt{n}$.}\COMMENT{The last inequality uses $\left.\binom{(1 - \delta)\sqrt n}{2}\middle/\binom{n}{2}\right. \geq \frac{(1 - \delta)^2n - (1 - \delta)\sqrt n}{n^2}  \geq \frac{1 - 2\delta + \delta^2 - (1 - \delta)/\sqrt n}{n} \geq \frac{1 - 2\delta}{n}$.}
\begin{align*}
    \vol_{\cH}(B) & \geq |B| \frac{k(k+1)}{n(n-1)} \geq \frac{|B|}{n}\:\:,\:\:
    \vol_{\cH}(A^+) = |A^+|\frac{k(k-1)}{n(n-1)} \geq \frac{|A^+|}{n} \left ( 1 - \frac{3}{\sqrt{n}} \right ),\\
    \vol_{\cH}(A^-) &\geq |A^-| \frac{(1-\delta)\sqrt{n} \cdot ((1-\delta)\sqrt{n} - 1)}{n(n-1)} \geq |A^-| \frac{1-2\delta}{n}. 
\end{align*}

Combining these inequalities with the fact that $\vol_{\cH}(B) + \vol_{\cH}(A^+) + \vol_{\cH}(A^-) = \vol_{\cH}(\cH) \leq t$ (since $\cH$ has maximum codegree at most $t$), we obtain
\begin{align}\label{eqn:abn}
    r = |A^-| + |A^+| + |B| - tn \leq \frac{3|A^+|}{\sqrt{n}} + 2 \delta |A^-| \leq 2\delta |A|.
\end{align}

If $|A| \leq t/(2\delta)$, then by Proposition~\ref{na-prop0}, for every distinct intersecting $e,f \in A$, $\{e, f\}$ is $t$-useful. Moreover, by~\eqref{eqn:abn}, $|A| \ge r/(2 \delta) \ge 6r$.  Thus, $A$ satisfies the conditions of Proposition~\ref{lem:chromindex1}\ref{hypo:A-bound}, so  $\chi'_\ell(\cH) < tn$, as desired. Now we may assume that
\begin{align}\label{eqn:bounda}
    |A| > t/(2\delta).
\end{align}

If $|A^-| \leq t/(4\delta)$, then by Proposition~\ref{na-prop}\ref{na-prop:k-1-bound}, for every distinct intersecting $e,f \in A$ with $|V(e) \cap V(f)| \leq \alpha k$, we have that $\{e, f\}$ is $t$-useful.  Moreover,~\eqref{eqn:abn} and~\eqref{eqn:bounda} together imply that $r = |A| + |B| - tn \leq 2\delta|A| \leq |A|/6 - t/(3\alpha^2)$. Thus, $A$ satisfies the conditions of Proposition~\ref{lem:chromindex1}\ref{hypo:A-bound-intersect}, so  $\chi'_\ell(\cH) < tn$, as desired.
Hence, we may assume that
\begin{equation}\label{eqn:bounda-}
    |A^-| > t/(4\delta).   
\end{equation}

Thus if $|A^+| \leq \alpha \sqrt{n} |A^-|$, then by~\eqref{eqn:abn} and~\eqref{eqn:bounda-}, we have $r = |A^-| + |A^+ \cup B| - tn \leq (3\alpha + 2\delta) |A^-| \leq |A^-|/6 - t/(3\alpha^2)$, so $|A^-| \geq 6r + 2t / \alpha^2$. 
Moreover, by Proposition~\ref{na-prop}\ref{na-prop:(k-1)-edge}, for every distinct intersecting $e, f \in A^-$ with $|V(e) \cap V(f)| \le \alpha k$, the pair $\{ e,f \}$ is $t$-useful. Hence, $A^-$ satisfies the conditions of Proposition~\ref{lem:chromindex1}\ref{hypo:A-bound-intersect}, so $\chi'_\ell(\cH) < tn$, as desired.
Thus we assume
\begin{equation}\label{eqn:bounda+}
    |A^+| > \alpha \sqrt{n}|A^-| \overset{\eqref{eqn:bounda-}}{> }\frac{\alpha t}{4\delta} \sqrt{n}.
\end{equation}

Now let $L\coloneqq L(\cH)$ be the line graph of $\cH$, and let $N$ be a maximal matching in $\overline L$.  We may assume $|N| \leq r$, as otherwise the lemma follows from Observation~\ref{lem:chromindex}.  Combining this inequality with~\eqref{eqn:abn} and~\eqref{eqn:bounda+}, we have
\begin{equation}\label{eqn:sizeN}
  |N| \leq r \leq \frac{3|A^+|}{\sqrt{n}} + 2 \delta |A^-| < \frac{3|A^+|}{\sqrt{n}} + \frac{2 \delta}{\alpha \sqrt{n}} |A^+| <  \frac{5|A^+|}{\sqrt{n}}.
\end{equation}

\begin{claim}\label{claim:useful-pairs}
    There are $2r + 2$ distinct $e_1, \dots, e_{2r + 2} \in A^{+} \setminus V(N)$ such that for each $i \in [r+1]$,
\begin{enumerate}[(i)]
\item\label{pairwise-small-intersect} $|V(e_{2i - 1})\cap V(e_{2i})| \leq \alpha k$ and 

\item\label{extremal-useful-pairs} there exists $w_i \in V(e_{2i - 1})\cap V(e_{2i})$ such that $|\{ e \in A^- : w_i \in V(e) \}| < t/(4\delta)$.
\end{enumerate}
\end{claim}
\begin{claimproof}
Let $V_{\rm bad} \coloneqq \{ x \in V(\cH) \: : \: |\{e \in A^- \: : \: x \in V(e) \}| \geq t/(4\delta) \}$. Let $\mathcal{P} \coloneqq \{(w,e) \: : \: \textrm{$w \in V_{\rm bad}$, $w \in V(e)$, and $e \in A^-$}\}$, 
and note that (by~\eqref{eqn:rangen}), $|V_{\rm bad}| \cdot t/(4\delta) \leq |\mathcal{P}| \leq (k-1) |A^-| \leq 2 \sqrt{n} |A^-|$, which implies
\begin{equation}\label{claim:nobad}
    |V_{\rm bad}| \leq 8 \delta |A^-| \frac{\sqrt{n}}{t}.
\end{equation}
Let $A_{\rm bad}^+ \coloneqq \{e \in A^+ \: : \: |V(e) \cap V_{\rm bad}| \geq \sqrt{\delta n} \}$. Let $\mathcal{Q} \coloneqq \{(w,e) \: : \: \textrm{$w \in V_{\rm bad}$, $w \in V(e)$, and $e \in A_{\rm bad}^+$}\}$. Thus, by Observation~\ref{obs:degree}, we have\COMMENT{we used $k-1 \ge \sqrt{n}-2$ below}
\begin{equation*}
|A_{\rm bad}^+| \sqrt{\delta n} \leq |\mathcal Q| \leq |V_{\rm bad}| \frac{t(n-1)}{k - 1} \overset{\eqref{eqn:rangen},\eqref{claim:nobad}}{\leq} 16 \delta n |A^-|,
\end{equation*}
so we have
\begin{equation}\label{eqn:bounda*}
    |A_{\rm bad}^+| \leq 16 \sqrt{\delta n} |A^-| \overset{\eqref{eqn:bounda+}}{\leq} \frac{16  \sqrt{\delta}}{\alpha}|A^+| < \frac{|A^+|}{20}.
\end{equation}
Thus,\COMMENT{by~\eqref{eqn:sizeN}, $|V(N)| = 2 |N| < \frac{10|A^+|}{\sqrt{n}}$}
\begin{equation}\label{eqn:bounda+_2}
|A^+ \setminus (A_{\rm bad}^+ \cup V(N))| \geq |A^+| - |A_{\rm bad}^+| - |V(N)| \overset{\eqref{eqn:sizeN},\eqref{eqn:bounda*}}{\geq} \frac{9|A^+|}{10}.
\end{equation}

Now we iterate the following procedure for $i \in [r+1]$. Suppose we have found distinct $e_1 , \dots , e_{2(i-1)} \in A^+ \setminus (A_{\rm bad}^+ \cup V(N))$ such that for each $j \in [i-1]$, $|V(e_{2j-1}) \cap V(e_{2j})| \leq \alpha k$ and $(V(e_{2j - 1})\cap V(e_{2j})) \setminus V_{\rm bad} \ne \varnothing$. Let
\begin{equation*}
    A_i \coloneqq A^+ \setminus \left(A_{\rm bad}^+ \cup V(N) \cup \{e_1 , \dots , e_{2(i-1)} \}\right).
\end{equation*}

Now we aim to find a pair $\{ e_{2i-1} , e_{2i} \} \subseteq A_i$ such that $|V(e_{2i-1}) \cap V(e_{2i})| \leq \alpha k$, and $(V(e_{2i - 1})\cap V(e_{2i})) \setminus V_{\rm bad} \ne \varnothing$. Note that
\begin{equation}
\label{eq:size:Si}
    |A_i| \geq |A^+ \setminus (A_{\rm bad}^+ \cup V(N))| - 2r \overset{\eqref{eqn:sizeN},\eqref{eqn:bounda+_2}}{\geq} \frac{9|A^+|}{10} - \frac{10|A^+|}{\sqrt{n}} \overset{\eqref{eqn:bounda+}}{\geq} \frac{\alpha t}{10 \delta} \sqrt{n}.
\end{equation}

Let $\mathcal{T}_i \coloneqq \{(w,e) : e \in A_i, \:w \in V(e)\setminus V_{\rm bad}\}$. Since $|V(e) \cap V_{\rm bad}| \le \sqrt{\delta n}$ for each $e \in A_i \subseteq A^{+} \setminus A_{\rm bad}^+$, we deduce the following. (Also note that $V(\cH) \setminus V_{\rm bad} \ne \varnothing$ since $|V(e) \setminus V_{\rm bad}| \geq (1-\delta-\sqrt{\delta})\sqrt{n}$ for each $e \in A_i$, and $A_i \ne \emptyset$ by \eqref{eq:size:Si}.)
\begin{align*}
    \frac{1}{|V(\cH) \setminus V_{\rm bad}|}\sum_{w \in V(\cH) \setminus V_{\rm bad}} d_{A_i}(w) &= \frac{|\mathcal{T}_i|}{|V(\cH) \setminus V_{\rm bad}|} \geq \frac{|\mathcal{T}_i|}{n}
    = \frac{1}{n} \sum_{e \in A_i} |V(e) \setminus V_{\rm bad}| \\&\geq \frac{|A_i| (k - \sqrt{\delta n})}{n} \overset{\eqref{eq:size:Si},\eqref{eqn:rangen}}{\geq} \frac{\alpha t \sqrt{n}/(10\delta)  \cdot (\sqrt{n}/2)}{n} \\
    & = \frac{\alpha t}{20 \delta} > \frac{2t}{\alpha^2}.
\end{align*}

Thus, there exists a vertex $w_i \in V(\cH) \setminus V_{\rm bad}$ with $d_{A_i}(w_i) > 2t/\alpha^2$, so there exists a pair $\{e_{2i-1},e_{2i} \} \subseteq A_i$ such that $w_i \in V(e_{2i-1}) \cap V(e_{2i})$ and $|V(e_{2i-1}) \cap V(e_{2i})| \leq \alpha k$ by Observation~\ref{obs:intersect}.\COMMENT{Let $e$ be any edge in $A_i$ that contains $w_i$.  By Observation~\ref{obs:intersect}, at most $2t / \alpha^2 - 1$ edges $f\neq e$ satisfy $|V(e) \cap V(f)| \geq \alpha k$, so there is an edge $f \in A_i$ containing $w_i$ such that $|V(e) \cap V(f)| < \alpha k$.} Moreover, since $w_i \in V(\cH) \setminus V_{\rm bad}$, $|\{ e \in A^- : w_i \in V(e) \}| < t/(4\delta)$ by the definition of $V_{\rm bad}$.

Thus, iterating the above procedure allows us to find distinct $e_1 , \dots , e_{2r+2} \in A^+ \setminus V(N)$ satisfying~\ref{pairwise-small-intersect} and~\ref{extremal-useful-pairs} of Claim~\ref{claim:useful-pairs}, as desired.
\end{claimproof}

Let $e_1 , \dots , e_{2r+2} \in A^+ \setminus V(N)$ be the distinct edges obtained from Claim~\ref{claim:useful-pairs}. Since $N$ is maximal, $\cH\setminus V(N)$ is a clique in $L$, so $e_1 , \dots , e_{2r+2}$ are pairwise intersecting. Moreover, by Proposition~\ref{na-prop}\ref{na-prop:k-1-bound}, $\{e_{2i - 1}, e_{2i}\}$ is $t$-useful for each $i \in [r+1]$. Thus by Proposition~\ref{lem:chromindex0}, we deduce that $\chi_\ell'(\cH) < tn$, completing the proof of Lemma~\ref{extremal-case-lemma}.
\end{proof}

\section{Proofs of Theorems~\ref{main-thm} and~\ref{main-thm-2}}\label{section:proof}
In this section, we prove Theorems~\ref{main-thm} (our main result) and \ref{main-thm-2} (the stability version). For this, we first introduce tools for colouring locally sparse graphs and hypergraphs in Section~\ref{subsec:probColouringPrelim}. In Section~\ref{subsec:Overviewofproofs}, we give an overview of how we combine these tools with the results of the previous sections to prove Theorems~\ref{main-thm} and~\ref{main-thm-2}. We then  prove Theorem~\ref{main-thm-2} in Section~\ref{subsec:proofofstability} and prove Theorem~\ref{main-thm} (whose proof relies on Theorem~\ref{main-thm-2}) in Section~\ref{subsec:proofofmainthm}.

\subsection{Probabilistic hypergraph colouring preliminaries}
\label{subsec:probColouringPrelim}
We will need the following fundamental result of
Kahn~\cite{kahn1996asymptotically}, which asymptotically determines the list chromatic index of a hypergraph of small codegree. A precursor of Theorem~\ref{thm:kahn} was his main tool in proving an asymptotic version of the Erd\H{o}s-Faber-Lov\'asz Conjecture~\cite{kahn1992coloring}.
We use Theorem~\ref{thm:kahn} to colour the ‘small’ edges of $\cH$.

\KahnTheorem*

We also use the following result of Molloy and Reed~\cite[Theorem~10.5]{MR02} on the list chromatic index of locally sparse graphs. 
The study of colouring locally sparse graphs was initiated by Molloy and Reed~\cite{MR1997} in their work on a conjecture of Erd\H{o}s and Ne\v{s}et\v{r}il on strong edge-colouring graphs of bounded maximum degree and in Reed's~\cite[Theorem 1]{reed1998} work on his $\omega$, $\Delta$, and $\chi$ conjecture~\cite[Conjecture 1]{reed1998}.
This result has recently been improved in~\cite{BPP18, BJ15, HdVK20}, but we do not need these stronger versions.

\begin{theorem}[Molloy and Reed~\cite{MR02}]\label{local-sparsity-theorem}
  Let $0 < 1/\Delta \ll \zeta \leq 1$.  Let $G$ be a graph with $\Delta(G) \leq \Delta$.  If every $v\in V(G)$ satisfies $e(G[N(v)]) \leq (1 - \zeta)\binom{\Delta}{2}$, then $\chi_\ell(G) \leq (1 - \zeta/e^6)\Delta$.
\end{theorem}

The following lemma will be used to show that the line graph of $\cW$, where $\cW$ is obtained from an application of Lemma~\ref{lemma:good-ordering} in the proof of Theorem~\ref{main-thm-2}, is locally sparse. Then we use Theorem~\ref{local-sparsity-theorem} to colour the edges of $\cW$ with $(1 - o(1)) tn$ colors.

\begin{lemma}\label{local-sparsity-lemma}
  Let $0 < 1/n_0, 1/r \ll \alpha \ll \zeta < 1$, let $n, t \in \mathbb N$ where $n \geq n_0$, and suppose $r \leq (1 - \zeta)\sqrt n$.  
  If $\cH$ is an $n$-vertex hypergraph with 
  $\Delta_2(\cH) \leq t$ and 
  $r \leq \antirank(\cH) \leq \rank(\cH) \leq (1 + \alpha)$r,
  then $\Delta(L) \leq \Delta$ and $e(L[N(e)]) \leq ( 1 - \frac{5 \zeta}{6} ) \binom{\Delta}{2}$ for each $e \in \cH$, where $\Delta \coloneqq (1 + \alpha)trn/(r - 1)$ and $L\coloneqq L(\cH)$.
\end{lemma}
\begin{proof}
  Since every edge of $\cH$ has size at least $r$ and $\cH$ has maximum codegree at most $t$, we have $d_{\cH}(v) \leq tn/(r-1)$ for every $v \in V(\cH)$ by Observation~\ref{obs:degree}. Thus, for every $e \in \cH$, we have $|N(e)| \leq \sum_{v \in V(e)}d_\cH(v) \leq (1+\alpha)trn/(r-1)$, hence, $\Delta(L) \leq \Delta$.  
  
  Now we aim to estimate $e(L[N(e)])$. First, partition $N(e)$ into $N_1 \coloneqq \{f \in N(e) : |V(e) \cap V(f)| \leq \alpha |V(e)|\}$ and $N_2 \coloneqq N(e) \setminus N_1$.  
  Since $\cH$ has maximum codegree at most $t$, for each $f \in N_1$, we have
  \begin{align*}
      |N(e) \cap N(f)| & \leq \sum_{v \in V(e) \cap V(f)} d_\cH(v) + t|V(e) \setminus V(f)||V(f) \setminus V(e)|\\ 
      & \leq \alpha(1 + \alpha)r \cdot \frac{tn}{r - 1}  + t(1 + \alpha)^2 r^2 \leq (2 \alpha + (1 + 3 \alpha)(1 - \zeta) ) tn\\
      & \leq \left( 1 - \frac{9\zeta}{10} \right) tn.
  \end{align*}
  By Observation~\ref{obs:intersect}, there are at most $2t/\alpha^2$ edges in $N_2$.
  Using this fact and the above inequality, we have\COMMENT{in the last inequality below we used $(1 - \frac{9\zeta}{10}) \frac{tn \Delta}{2} = (1 - \frac{9\zeta}{10}) \frac{ \Delta}{2} \cdot \Delta \frac{r-1}{r(1+ \alpha)} \le (1 - \frac{9\zeta}{10}) (\binom{\Delta}{2}+ \Delta) \leq (1 - \frac{8 \zeta}{9}) \binom{\Delta}{2}$}
  \begin{align*}
      e(L[N(e)])
      &\leq \sum_{f \in N_1} \frac{1}{2}|N(e) \cap N(f)| + \sum_{f \in N_2}|N(e) \cap N(f)|\\
      & \leq \left (1 - \frac{9\zeta}{10} \right ) \frac{tn}{2}|N_1| + |N_2|\cdot\Delta \leq \left (1 - \frac{9\zeta}{10} \right ) \frac{tn \Delta}{2} + \frac{2t}{\alpha^2} \cdot \Delta  \leq \left ( 1 - \frac{5 \zeta}{6} \right ) \binom{\Delta}{2},
  \end{align*}
  as desired.
\end{proof}


\begin{corollary}\label{sparsity-corollary}
  Let $0 < 1/n_0, 1/r \ll \alpha \ll \zeta < 1$, let $n, t \in \mathbb N$ where $n \geq n_0$, and suppose $r \leq (1 - \zeta)\sqrt n$.  
  If $\cH$ is an $n$-vertex hypergraph with 
  $\Delta_2(\cH) \leq t$ and $r \leq \antirank(\cH) \leq \rank(\cH) \leq (1 + \alpha)$r,
  then $\chi'_\ell(\cH) \leq (1 - \zeta/500)tn$.
\end{corollary}

We need the following proposition in the proofs of Theorems~\ref{main-thm} and~\ref{main-thm-2}.
\begin{proposition}\label{lemma:reserve-colors}
    Let $1 / n_0 \ll \xi < 1$, let $\gamma \in [0, 1]$, and let $t, n \in \mathbb N$ where $n \geq n_0$.  If $\cH$ is an $n$-vertex hypergraph with $\Delta_2(\cH) \leq t$
    and $C$ is a list assignment for $\cH$ satisfying $|C(e)| \geq tn / 2$ for every $e \in \cH$, then there exists a set $R \subseteq \bigcup_{e \in \cH}C(e)$ such that every $e \in \cH$ satisfies $|R \cap C(e)| = (\gamma \pm \xi)|C(e)|$.
\end{proposition}
This proposition can be easily proved by choosing each colour in $\bigcup_{e \in \cH}C(e)$ in $R$ randomly and independently with probability $\gamma$ and showing that $R$ satisfies the lemma with high probability, so we omit the details.\COMMENT{\begin{proof}
 Include every colour $c \in \bigcup_{e \in \cH}C(e)$ in $R$ randomly and independently with probability $\gamma$.  By a standard application of the Chernoff bound, every $e \in \cH$ satisfies $|R \cap C(e)| = (\gamma \pm \xi) |C(e)|$ with probability at least $1 - 2\exp(- \Omega(|C(e)|)) \geq 1 - 2\exp(-\Omega(n))$.  Since $e(\cH) \leq tn^2$, the result follows by the Union Bound.
\end{proof}}

\subsection{Overview of the proofs of Theorems~\ref{main-thm} and~\ref{main-thm-2}}
\label{subsec:Overviewofproofs}

We now have all of the ingredients needed to prove Theorems~\ref{main-thm} and~\ref{main-thm-2}.  For the reader's convenience, we provide a brief overview of both proofs in the case when all edges have the same list of available colours; that is, we overview the proofs of both results where we replace ``$\chi'_\ell$" with ``$\chi'$".

Let constants $n_0$, $r_0$, $r_1$, $\mu$ $\delta$, $\gamma$, $\eps$ satisfy
\begin{equation*}
    0 < 1 / n_0 \ll 1 / r_0 \ll 1 / r_1 \ll \mu \ll \delta \ll \gamma \ll \eps \ll 1.
\end{equation*}
In both proofs, we decompose $\cH$ into spanning subhypergraphs $\cH_{\mathrm{sml}} \coloneqq \{ e \in \cH \: : \: |V(e)| \leq r_1 \}$, $\cH_{\mathrm{med}} \coloneqq \{ e \in \cH \: : \: r_1 < |V(e)| \leq r_0 \}$, and $\cH_{\mathrm{lrg}} \coloneqq \{ e \in \cH \: : \: |V(e)| > r_0 \}$, we colour $\cH_{\mathrm{lrg}}$ before colouring $\cH_{\mathrm{sml}}$, and we use disjoint sets of colours on $\cH_{\mathrm{sml}}$ and $\cH_{\mathrm{med}}$.

We begin with the proof of Theorem~\ref{main-thm-2} (the stability theorem), as we use this result in the proof of Theorem~\ref{main-thm}.  In this proof, we apply Lemma~\ref{lemma:good-ordering} to $\cH_{\mathrm{lrg}}$ and obtain a partition of it into $\cH_1$, $\cW$, and $\cH_2$ and an ordering $\ordering$ of its edges.  We use Corollary~\ref{sparsity-corollary} to colour the edges of $\cW$, and we use~\ref{reordering-goodness2} to colour $\cH_2$ (greedily, in the ordering provided by $\ordering$) with a disjoint set of colours, so that in total we use at most $(1 - 2\mu)tn + 1$ colours. 
We then use~\ref{reordering-goodness1} and~\ref{partition-order} to extend this colouring to $\cH_1$ without using any additional colours. Finally, we use Theorem~\ref{thm:kahn} to extend this colouring to $\cH_{\mathrm{sml}}$ with no additional colours and also to colour $\cH_{\mathrm{med}}$ with a small set of additional colours, so that in total, we use at most $(1 - \mu)tn$ colours.

In the proof of Theorem~\ref{main-thm}, by Theorem~\ref{main-thm-2}, we may assume that at least $(1 -  \delta)tn$ edges have size $(1 \pm \delta)\sqrt n$.  Hence, we can apply Lemma~\ref{lemma:good-ordering-extremal} to $\cH_{\mathrm{lrg}} \cup \cH_{\mathrm{med}}$ to obtain a partition of it into $\cH_1$, $\cH_2$, and $\cH_3$ and an ordering $\ordering$ of its edges.  We use Lemma~\ref{extremal-case-lemma} to colour the edges of $\cH_3$, and we use~\ref{reordering-goodness3}--\ref{partition-order2} to extend this colouring to $\cH_{\mathrm{lrg}}\cup\cH_{\mathrm{med}}$ without using any additional colours while also ensuring that only a small subset of colours are used on $\cH_{\mathrm{med}}$.  As in the proof of Theorem~\ref{main-thm-2}, we can extend the colouring to $\cH_{\mathrm{sml}}$ using Theorem~\ref{thm:kahn} without using any additional colours.

\subsection{Proof of the stability theorem, Theorem~\ref{main-thm-2}}
\label{subsec:proofofstability}

We begin with the proof of Theorem~\ref{main-thm-2}, but first, we prove it in the special case when all edges are large in the following lemma.

\begin{lemma}\label{lemma:large-extremal-col}
    Let $0 < 1/n_0 \ll 1/r \ll \mu \ll \delta \ll 1$, and let $n, t\in\mathbb N$ where $n \geq n_0$.  Let $\mathcal H$ be an $n$-vertex hypergraph where $\cH$ has $\Delta_2(\cH) \leq t$ and $\antirank(\cH) > r$.
    If the number of edges in $\cH$ of size $(1 \pm \delta)\sqrt{n}$ is at most $(1-\delta)tn$, then $\chi'_\ell(\cH) \leq (1 - \mu)tn$.
\end{lemma}
\begin{proof}
 We will instead prove the contrapositive statement that if $\chi'_\ell(\cH) > (1 - \mu)tn$, then the number of edges in $\cH$ of size $(1 \pm \delta)\sqrt{n}$ is more than $(1-\delta)tn$.
 
 Let $\delta'$ satisfy $\mu \ll \delta' \ll \delta$. 
 Since $\chi'_\ell(\cH) > (1 - \mu)tn$, there is a list assignment $C$ for the edges of $\cH$ satisfying $|C(e)| = \lceil (1 - \mu)tn\rceil$ for every $e \in \cH$ such that $\cH$ is not $C$-edge-colourable.
 By Lemma~\ref{lemma:good-ordering} (with $r$, $\mu$, and $\delta'$ playing the roles of $r$, $\mu$, and $\delta$, respectively) there is a partition of $\cH$ into three spanning subhypergraphs $\cH_1$, $\cW$, and $\cH_2$ satisfying
  \begin{enumerate}[label=(P\arabic*')]
  \item\label{W-max-edge-size'} $\rank(\cW) \leq (1 + \delta')\antirank(\cW)$,
  \item\label{W-volume2'} $\vol_\cH(\cW) \geq (1 - \delta')t$,
  \end{enumerate}
  and~\ref{H2-edge-size} and an ordering $\ordering$ of the edges of $\cH$ satisfying~\ref{reordering-goodness1},~\ref{reordering-goodness2}, and~\ref{partition-order}.  
  
  Let $r' \coloneqq \antirank(\cW)$, and let $\zeta \coloneqq 1 - r' / \sqrt n$.  First we claim that $r' \geq (1 - \delta)\sqrt n$.  Suppose to the contrary.  We derive a contradiction by showing $\cH$ is $C$-edge-colourable.  To that end, we apply Proposition~\ref{lemma:reserve-colors} with $\zeta / 900$ and $\zeta / 10000$ playing the roles of $\gamma$ and $\xi$, respectively, to obtain a set $R \subseteq \bigcup_{e \in \cH}C(e)$ such that
  \begin{equation}\label{large-reserved-colours}
    \text{every $e \in \cH$ satisfies $ \zeta tn / 1000 \leq |R \cap C(e)| \leq \zeta tn / 800.$}
  \end{equation}
  For every $e \in \cH$, let $C_1(e) \coloneqq C(e) \setminus R$ and $C_2(e) \coloneqq C(e) \cap R$.  
  
   Since we assume $r' < (1 - \delta)\sqrt n$, we have $\zeta > \delta \gg \mu$, so $\zeta / 800 + \mu  < \zeta / 500$.  Hence, by~\eqref{large-reserved-colours}, every $e \in \cH$ satisfies $|C_1(e)| \geq (1 - \mu)tn - \zeta tn / 800 > (1 - \zeta / 500)tn$. 
   Therefore, by~\ref{W-max-edge-size'} and Corollary~\ref{sparsity-corollary} with $r'$ and $\delta'$ playing the roles of $r$ and $\alpha$, respectively, we have $\chi'_\ell(\cW) \leq (1 - \zeta / 500)tn$, which implies that there is a $C_1$-edge-colouring $\phi_1$ of $\cW$. 
 
   Next we show that there is a $C_2$-edge-colouring $\phi_2$ of $\cH_2$.  To that end, let $k \coloneqq e(\cH_2)$.  By~\eqref{large-reserved-colours}, every $e \in \cH_2$ satisfies $|C_2(e)| \geq \zeta tn / 1000$.  Clearly $\chi'_\ell(\cH_2) \leq k$, so we assume $k > \zeta tn/ 1000$.  Since $\zeta > \delta$, we have $k > \zeta tn/ 1000 > 2\delta^2 tn$.  By~\ref{H2-edge-size}, every edge of $\cH_2$ has size at least $r'$, so we have $\vol_{\cH}(\cH_2) \geq k(r' - 1)^2 / n^2$.  On the other hand, by~\ref{W-volume2'}, and since $\cH_2 \cap \cW = \varnothing$, we have $\vol_{\cH}(\cH_2) \leq \delta' t \leq \delta^3 t$.
   Thus, $2\delta^2 tn < k \leq \delta^3 tn^2 / (r' - 1)^2$, so $r' < \delta^{1/4} \sqrt{n}$.
   Therefore, $\zeta > 1000/1001$, so $tn/2000 + 1 \leq \zeta tn/ 1000$.  Hence, by~\ref{reordering-goodness2}, we can properly colour $\cH_2$ greedily from $C_2$, in the ordering provided by $\ordering$, as claimed.

  By the choice of $C_1$ and $C_2$, we can combine $\phi_1$ and $\phi_2$ to obtain a $C$-edge-colouring of $\cH_2 \cup \cW$, and by~\ref{reordering-goodness1} and~\ref{partition-order}, we can extend such a colouring to $\cH_1$ greedily without using any additional colours, contradicting that $\cH$ is not $C$-edge-colourable.  Therefore $r' \geq (1 - \delta)\sqrt n$, as claimed.
  
  We also have $r' \leq \sqrt{n / (1 - 4\mu)}$, as otherwise~\ref{H2-edge-size} and the fact that $\vol_\cH(\cH_2\cup \cW) \leq t$ together imply $e(\cH_2\cup \cW) \leq (1 - 2\mu)tn$.\COMMENT{We have $t \geq \vol_{\cH}(\cH_2\cup \cW) \geq \left.e(\cH_2 \cup \cW)\binom{r'}{2}\middle/\binom{n}{2}\right.$, and if $r' \geq \sqrt{n / (1 - 4\mu)}$, then $\left.\binom{r'}{2}\middle/\binom{n}{2}\right. \geq (1 - 1/\sqrt n)/((1 - 4\mu)n)$.}  
  Together with~\ref{reordering-goodness1} and ~\ref{partition-order}, this fact would imply that every $e \in \cH$ satisfies $\fwddeg_\cH(e) \leq (1 - 2\mu) tn$, in which case $\chi'_\ell(\cH) \leq (1 - \mu) tn$, a contradiction.
  
   Since $\sqrt{n / (1 - 4\mu)} \geq r' \geq (1 - \delta)\sqrt n$, by~\ref{W-max-edge-size'}, the edges in $\cW$ have size $(1\pm \delta)\sqrt n$.  In fact, the edges in $\cW$ have size at most $(1 + \delta^2)\sqrt n$, so by~\ref{W-volume2'}, we have $e(\cW) \geq \vol_\cH(\cW)(n - 1) / (1 + \delta^2)^2 > (1 - \delta)tn$,\COMMENT{since $\vol_\cH(\cW) \leq e(\cW)\left.\binom{(1 + \delta^2)\sqrt n}{2}\middle / \binom{n}{2}\right. \leq e(\cW) (1 + \delta^2)^2 / (n - 1)$} as desired.
\end{proof}

Now we restate and prove Theorem~\ref{main-thm-2}.
\mainthmsecond*
\begin{proof}
Without loss of generality, we may assume that $\delta$ is sufficiently small. Let 
\begin{equation*}
    0 < 1/n_0 \ll 1/r_0 \ll 1/r_1 \ll \mu \ll \delta \ll 1,
\end{equation*}
let $t, n \in \mathbb N$ where $n \geq n_0$, and let $\cH$ be an $n$-vertex hypergraph with maximum codegree at most $t$ and maximum degree at most $(1 - \delta)tn$ such that $|\{e \in \cH : |V(e)| = (1 \pm \delta)\sqrt n\}| \leq (1 - \delta)tn$.  We will show that $\chi'_\ell(\cH) \leq (1 - \mu)tn$, so let $C$ be any assignment of lists for $\cH$ satisfying $|C(e)| = \lceil(1 - \mu)tn\rceil$ for every $e \in \cH$. We will show that $\cH$ is $C$-edge-colourable.  
  
 Decompose $\cH$ into the following spanning subhypergraphs:
  \begin{itemize}
  \item $\cH_{\mathrm{sml}} \coloneqq \{ e \in \cH \: : \: |V(e)| \leq r_1 \}$,
  \item $\cH_{\mathrm{med}} \coloneqq \{ e \in \cH \: : \: r_1 < |V(e)| \leq r_0 \}$, and 
  \item $\cH_{\mathrm{lrg}} \coloneqq \{ e \in \cH \: : \: |V(e)| > r_0 \}$.
  \end{itemize}
  Note that, since $\cH$ has maximum codegree at most $t$, by Observation~\ref{obs:degree}
  \begin{enumerate}[label=(\theequation)]
    \stepcounter{equation}
  \item\label{eqn:med-max-deg} $\cH_{\mathrm{med}}$ has maximum degree at most $2tn / r_1 \leq \mu t n / 2$ and
    \stepcounter{equation}
  \item\label{eqn:sml-lrg-intersect2} every $e \in \cH_{\mathrm{sml}}$ satisfies $|N_\cH(e) \cap \cH_{\mathrm{lrg}}| \leq 2r_1 tn / r_0 \leq \mu tn$.
  \end{enumerate}
  
By Proposition~\ref{lemma:reserve-colors} with $2\mu$ and $\mu/2$ playing the roles of $\gamma$ and $\xi$, respectively, there exists a set $R \subseteq \bigcup_{e\in\cH}C(e)$ such that
  \begin{equation}\label{list-reserved-colours-stability}
    \text{every $e \in \cH$ satisfies $ \mu tn \leq |R \cap C(e)| \leq 3\mu tn.$}
  \end{equation}
For every $e \in \cH_{\mathrm{med}}\cup\cH_{\mathrm{sml}}$, let $C_1(e) \coloneqq C(e) \cap R$, and for every $e \in \cH_{\mathrm{lrg}}$, let $C_1(e) \coloneqq C(e) \setminus R$. For every $e \in \cH$, let $C_2(e) \coloneqq C(e) \setminus C_1(e)$.
  By~\eqref{list-reserved-colours-stability},
\begin{enumerate}[label=(\theequation)]
    \stepcounter{equation}
    \item\label{eqn:large-col-lb2} every $e\in\cH_{\mathrm{lrg}}$ satisfies $|C_1(e)| \geq (1 - \mu)tn - 3\mu tn \geq (1 - 4\mu)tn$,
    \stepcounter{equation}
  \item\label{eqn:med-col-lb2} every $e \in \cH_{\mathrm{med}}$ satisfies $|C_1(e)| \geq \mu tn$, and \stepcounter{equation}
  \item\label{eqn:small-col-lb2} every $e \in \cH_{\mathrm{sml}}$ satisfies $|C_2(e)| \geq (1 - 4\mu)tn$.
\end{enumerate}

By~\ref{eqn:large-col-lb2} and Lemma~\ref{lemma:large-extremal-col} with $r_0$, $4\mu$, and $\delta$ playing the roles of $r$, $\mu$, and $\delta$, respectively, there is a $C_1$-edge-colouring $\phi_{\mathrm{lrg}}$ of $\cH_{\mathrm{lrg}}$.
By~\ref{eqn:med-col-lb2},~\ref{eqn:med-max-deg}, and Theorem~\ref{thm:kahn} with $r_0$, $1$, $2/(\mu n_0)$, and $\mu tn / 2$ playing the roles of $k$, $\eps$, $\delta$, and $D$, respectively, there is a $C_1$-edge-colouring $\phi_{\mathrm{med}}$ of $\cH_{\mathrm{med}}$.  By the choice of $C_1$, we can combine $\phi_{\mathrm{lrg}}$ and $\phi_{\mathrm{med}}$ to obtain a $C_1$-edge-colouring $\phi_1$ of $\cH_{\mathrm{lrg}}\cup\cH_{\mathrm{med}}$.

Now for each $e \in \cH_{\mathrm{sml}}$, let $C'_2(e) \coloneqq C_2(e)\setminus \{\phi_1(f) : f \in N_\cH(e) \cap \cH_{\mathrm{lrg}}\}$.  By~\ref{eqn:sml-lrg-intersect2} and~\ref{eqn:small-col-lb2}, $|C'_2(e)| \geq (1 - 5\mu)tn \geq (1 + \mu)(1 - \delta) tn$ for every $e \in \cH_{\mathrm{sml}}$.  Therefore, by Theorem~\ref{thm:kahn} with $r_1$, $\mu$, $2 / n_0$, and $(1 - \delta)tn$ playing the roles of $k$, $\eps$, $\delta$, and $D$, respectively, there exists a proper edge-colouring $\phi_2$ of $\cH_{\mathrm{sml}}$ such that $\phi_{2}(e) \in C'_2(e)$ for every $e \in \cH_{\mathrm{sml}}$.  By combining $\phi_1$ and $\phi_2$, we obtain the desired $C$-edge-colouring of $\cH$.
\end{proof}

\subsection{Proof of Theorem~\ref{main-thm}}
\label{subsec:proofofmainthm}
We conclude this section by restating and proving  Theorem~\ref{main-thm}, which is our main result.
\mainthm*
\begin{proof}
 Let
  \begin{equation*}
    1 / n_0 \ll 1 / r_0 \ll 1 / r_1 \ll \delta \ll \gamma \ll \eps < 1,
  \end{equation*}
  let $n,t \in \mathbb N$ where $n \geq n_0$, and let $\cH$ be an $n$-vertex hypergraph with maximum codegree at most $t$ and maximum degree at most $(1 - \eps)tn$.  First, we show that $\chi'_\ell(\cH) \leq tn$, for which it suffices to show that if $C$ is an assignment of lists such that every $e \in \cH$ satisfies $|C(e)| = tn$, then $\cH$ has a proper edge-colouring $\phi$ such that $\phi(e) \in C(e)$ for every $e \in \cH$.  By Theorem~\ref{main-thm-2}, we may assume that the number of edges of size $(1 \pm \delta)\sqrt n$ in $\cH$ is at least $(1 - \delta)tn$.
  
   Decompose $\cH$ into the following spanning subhypergraphs:
  \begin{itemize}
  \item $\cH_{\mathrm{sml}} \coloneqq \{ e \in \cH \: : \: |V(e)| \leq r_1 \}$,
  \item $\cH_{\mathrm{med}} \coloneqq \{ e \in \cH \: : \: r_1 < |V(e)| \leq r_0 \}$, and 
  \item $\cH_{\mathrm{lrg}} \coloneqq \{ e \in \cH \: : \: |V(e)| > r_0 \}$.
  \end{itemize}
  By Lemma~\ref{lemma:good-ordering-extremal}, there is a partition of $\cH_{\mathrm{med}} \cup \cH_{\mathrm{lrg}}$ into three spanning subhypergraphs $\cH_1$, $\cH_2$, and $\cH_3$ satisfying~\ref{H1-edge-size} and~\ref{W-min-edge-size} and an ordering $\ordering$ of the edges of $\cH$ satisfying~\ref{reordering-goodness3}--\ref{partition-order2}. Since $\cH$ has maximum codegree at most $t$, note that by Observation~\ref{obs:degree}
  \begin{equation}
      \label{eqn:sml-lrg-intersect}
      \text{every $e \in \cH_{\mathrm{sml}}$ satisfies $|N_\cH(e) \cap \cH_{\mathrm{lrg}}| \leq 2r_1 tn / r_0 \leq \gamma tn$.}
  \end{equation}

By Proposition~\ref{lemma:reserve-colors} with $2\gamma$ and $\gamma/2$ playing the roles of $\gamma$ and $\xi$, respectively, there exists a set $R \subseteq \bigcup_{e\in\cH}C(e)$ such that
  \begin{equation}\label{list-reserved-colours}
    \text{every $e \in \cH$ satisfies $ 3\gamma tn/2 \leq |R \cap C(e)| \leq 5\gamma tn/2$.}
  \end{equation}
  
For every $e \in \cH_{\mathrm{med}}\cup\cH_{\mathrm{sml}}$, let $C_1(e) \coloneqq C(e) \cap R$, and for every $e \in \cH_{\mathrm{lrg}}$, let $C_1(e) \coloneqq C(e)$.  
Let $C_2(e) \coloneqq C(e) \setminus C_1(e)$ for every $e \in \cH$.
  By~\eqref{list-reserved-colours},
\begin{enumerate}[label=(\theequation)]
    \stepcounter{equation}
  \item\label{eqn:large-col-lb} every $e\in\cH_{\mathrm{lrg}}$ satisfies $|C_1(e)| \geq tn$,
    \stepcounter{equation}
  \item\label{eqn:med-col-lb} every $e \in \cH_{\mathrm{med}}$ satisfies $|C_1(e)| \geq 3\gamma tn/2$, and \stepcounter{equation}
  \item\label{eqn:small-col-lb} every $e \in \cH_{\mathrm{sml}}$ satisfies $|C_2(e)| \geq (1 - 3\gamma)tn$.
\end{enumerate}

By~\ref{W-min-edge-size} (of Lemma~\ref{lemma:good-ordering-extremal}), and Lemma~\ref{extremal-case-lemma} (with $2\delta$ playing the role of $\delta)$), $\chi'_\ell(\cH_3) \leq tn$, and by~\ref{reordering-goodness3} and~\ref{partition-order2} (of Lemma~\ref{lemma:good-ordering-extremal}), $\chi'_\ell(\cH_3 \cup \cH_2) \leq tn$. By~\ref{H1-edge-size}, we have $\cH_{\mathrm{med}} \subseteq \cH_1$ and $\cH_3 \cup \cH_2 \subseteq \cH_{\mathrm{lrg}}$. Thus by~\ref{eqn:large-col-lb} and~\ref{eqn:med-col-lb}, for every $e \in \cH_3 \cup \cH_2$, $|C_1(e)| \ge tn$ and for every $e \in \cH_1$, $|C_1(e)| \ge 3\gamma tn/2$. Combining this with~\ref{reordering-goodness4},~\ref{partition-order2} and the fact that $\chi'_\ell(\cH_3 \cup \cH_2) \leq tn$, we obtain a $C_1$-edge-colouring $\phi_1$ of $\cH_3 \cup \cH_2 \cup \cH_1 = \cH_{\mathrm{lrg}} \cup \cH_{\mathrm{med}}$ such that $\phi_1(e) \in C_1(e)$ for every $e \in \cH_{\mathrm{lrg}} \cup \cH_{\mathrm{med}}$ (by colouring $\cH_3 \cup \cH_2$ first and then extending it greedily to colour $\cH_1$). 

Now for each $e \in \cH_{\mathrm{sml}}$, let $C'_2(e) \coloneqq C_2(e)\setminus \{\phi_1(f) : f \in N_\cH(e) \cap \cH_{\mathrm{lrg}}\}$.  By~\eqref{eqn:sml-lrg-intersect} and~\ref{eqn:small-col-lb}, $|C'_2(e)| \geq (1 - 4\gamma)tn \geq (1 + \gamma)(1 - \eps) tn$ for every $e \in \cH_{\mathrm{sml}}$.  Therefore, by Theorem~\ref{thm:kahn} with $r_1$, $\gamma$, $2 / n_0$, and $(1 - \eps)tn$ playing the roles of $k$, $\eps$, $\delta$, and $D$, respectively, there exists a proper edge-colouring $\phi_2$ of $\cH_{\mathrm{sml}}$ such that $\phi_{2}(e) \in C'_2(e)$ for every $e \in \cH_{\mathrm{sml}}$.  By combining $\phi_1$ and $\phi_2$, we obtain the desired $C$-edge-colouring of $\cH$.  

It remains to show that if $\chi'_\ell(\cH) = tn$, then $\cH_3$ is intersecting and has exactly $tn$ edges, in which case Theorem~\ref{thm:char_intersect} implies\COMMENT{Suppose we showed that $\cH_3$ is intersecting. If there is a singleton edge, then all of the $tn$ edges in $\cH_3$ must contain that singleton edge (since $\cH_3$ is intersecting), so there is a vertex of degree $tn$, contradicting our assumption that the maximum degree of $\cH$ is at most $(1-\eps)tn$. Therefore, $\cH_3$ has no edge of size one, and Theorem~\ref{thm:char_intersect} applies.} that $\cH_3$ is a $t$-fold projective plane (since $\cH_3$ cannot be a $t$-fold near-pencil by our assumption that the maximum degree of $\cH$ is at most $(1-\eps)tn$), so $\cH = \cH_3$, as desired. To that end it suffices to show that if $\cH_3$ is not an intersecting hypergraph with $tn$ edges, and $C$ is an assignment of lists $C(e)$ to every $e\in \cH$, such that every $e\in\cH$ satisfies $|C(e)| = tn - 1$, then $\cH$ is $C$-edge-colourable.  The proof is nearly identical to the previous argument, replacing~\ref{eqn:large-col-lb} with
\begin{enumerate}[label=\ref*{eqn:large-col-lb}']
  \item every $e\in\cH_{\mathrm{lrg}}$ satisfies $|C_1(e)| \geq tn - 1$,
\end{enumerate}
so we omit the details.\COMMENT{Indeed, note that if $\cH$ is not intersecting, Lemma~\ref{extremal-case-lemma} shows that $\chi'(\cH_3) \leq tn-1$ and by~\ref{reordering-goodness3} and~\ref{partition-order2} (of Lemma~\ref{lemma:good-ordering-extremal}), $\chi'_\ell(\cH_3 \cup \cH_2) \leq tn-1$. The rest of the argument is identical.} (Note that in this case $\chi'_\ell(\cH_3 \cup \cH_2) \leq tn-1$ by Lemma~\ref{extremal-case-lemma} and by~\ref{reordering-goodness3}, and~\ref{partition-order2} of Lemma~\ref{lemma:good-ordering-extremal}.)
\end{proof}

\section{Intersecting hypergraphs with bounded codegree: Proof of Theorem~\ref{thm:char_intersect}}\label{sec:char}
In this section we prove Theorem~\ref{thm:char_intersect}.  Apart from being of independent interest, it is also used in the proof of the ``moreover" part of Theorem~\ref{main-thm}. As stated in Theorem~\ref{thm:charlinear}, de Bruijn and Erd\H{o}s~\cite{bruijn_erdos1948} and F\"uredi~\cite{furedi1986} characterised those intersecting linear hypergraphs whose number of
edges is maximum. Theorem~\ref{thm:char_intersect} is a version of this result for hypergraphs with maximum codegree at most $t$.\COMMENT{Let us derive Theorem~\ref{thm:charlinear} from Theorem~\ref{thm:char_intersect}. Assume $\cH$ is a linear intersecting hypergraph with no edge of size one. Then by Theorem~\ref{thm:char_intersect} with $t = 1$, $e(\cH) \leq \max_{v} |N[v]| \leq n$. Moreover, if $e(\cH) =  n$, then there is a vertex $v$ such that $e(\cH) = |N[v]| = n$ (i.e., $N[v]$ spans all vertices of $\cH$), so $\cH[N[v]] = \cH$ is a projective plane or a near-pencil, proving Theorem~\ref{thm:charlinear}.} 

\begin{theorem}[de Bruijn and Erd\H{o}s~\cite{bruijn_erdos1948}, F\"uredi~\cite{furedi1986}]\label{thm:charlinear}
If $\cH$ is an $n$-vertex linear intersecting hypergraph with no edge of size one, then $e(\cH) \leq n$. Moreover, if $e(\cH) = n$, then $\cH$ is either
\begin{itemize}
    \item a projective plane of order $k$ (so there exists $k \in \mathbb{N}$ with $n = k^2 + k + 1$) or
    
    \item a near-pencil.
\end{itemize}
\end{theorem}

We will use the following proposition whose proof is nearly identical to the proof of~\cite[Lemma 2.1]{KS1992}. Since it is short, we include it here for completeness.
\begin{proposition}\label{prop:motzkin}
    If $G$ is a simple bipartite graph with bipartition $(X, Y)$ such that no vertex in $X$ is adjacent to every vertex in $Y$, then  
    \begin{equation}\label{eqn:motzkinineq_prop}
        \sum\left(\frac{|X|d(x) - |Y|d(y)}{(|X| - d(y))(|Y| - d(x))} : x \in X,~y\in Y,~xy\notin E(G)\right) \geq 0.
    \end{equation}
\end{proposition}
\begin{proof}
The left side of~\eqref{eqn:motzkinineq_prop} is equal to $\sum\left ( \frac{|Y|}{|Y| - d(x)} - \frac{|X|}{|X| - d(y)} : x \in X~y\in Y~xy\notin E(G)\right )$. Since every vertex $x \in X$ satisfies $d(x) < |Y|$, this is equal to $\sum_{x \in X} |Y| - \sum_{y \in Y : d(y) \ne |X|}  |X| \geq |X||Y| - |Y||X| = 0$, as desired.
\end{proof}

Now we prove Theorem~\ref{thm:char_intersect}.  Its proof is inspired by some ideas in~\cite{KS1992}.

\CharIntersect*
\begin{proof}
We may assume that $\cH \ne \varnothing$ and that every vertex has degree at least one. First, suppose there exists a vertex $v \in V(\cH)$ with degree $e(\cH)$. Then since $\cH$ has maximum codegree at most $t$,

\begin{equation*}
    e(\cH) = d(v) \leq \sum_{w \in N[v] \setminus \{ v \}} d(v,w) \leq t(|N[v]| - 1).
\end{equation*}
Therefore, we may assume that there is no vertex with degree $e(\cH)$, i.e., $d(x) < e(\cH)$ for every $x \in V(\cH)$.

\begin{claim}\label{claim:motzkin}
For every $v \in V(\cH)$,
\begin{equation}\label{eqn:motzkin}
    \sum_{x \in N[v]} \sum_{e \in \cH : x \notin V(e)} \frac{|N[v]|\cdot d(x) - e(\cH)\cdot |N[v] \cap V(e)| }{ (|N[v]| - |N[v] \cap V(e)|)(e(\cH) - d(x))} \geq 0.
\end{equation}
\end{claim}
\begin{claimproof}
Since there is no vertex $x \in V(\cH)$ with $d(x) = e(\cH)$, 
the claim follows immediately from Proposition~\ref{prop:motzkin}, where
$N[v]$ and $\cH$ play the roles of $X$ and $Y$, respectively, and $G$ is the bipartite graph with bipartition $(X, Y)$ where $xe \in E(G)$ if and only if $x \in V(e)$.
\end{claimproof}

\begin{claim}\label{claim:internal-degree-bound}
If $v_{\max} \in V(\cH)$ has maximum degree in $\cH$, $x \in N[v_{\max}]$, and $e \in \cH$ such that $x \notin V(e)$, then
\begin{equation*}
    d(x) \leq t \cdot |N[v_{\max}] \cap V(e)|.
\end{equation*}
\end{claim}
\begin{claimproof}
If $v_{\max} \notin V(e)$, then since $\cH$ is intersecting with maximum codegree at most $t$, $d(x) \leq d(v_{\max}) \leq \sum_{w \in V(e)} d(v_{\max},w) \leq t \cdot |N[v_{\max}] \cap V(e)|$. Otherwise, if $v_{\max} \in V(e)$, then $ N[v_{\max}] \cap V(e)= V(e)$, and we have $d(x) \leq \sum_{w \in V(e)}d(x,w) \leq t \cdot |V(e)| = t \cdot |N[v_{\max}] \cap V(e)|$,\COMMENT{since $\cH$ has maximum codegree at most $t$} proving the claim.
\end{claimproof}

\begin{claim}\label{claim:edge-bound}
    If $v_{\max} \in V(\cH)$ has maximum degree in $\cH$, then $e(\cH) \leq t \cdot |N[v_{\max}]|$.
\end{claim}
\begin{claimproof}
We apply Claim~\ref{claim:motzkin} with $v_{\max}$ playing the role of $v$. 
Since $v_{\max} \in N[v_{\max}]$, and every $x \in N[v_{\max}]$ satisfies $d(x) < e(\cH)$, the left side of~\eqref{eqn:motzkin} is not an empty summation. Not all terms in this summation can be negative, so there exist $x \in N[v_{\max}]$ and $e \in \cH$ with $x \notin V(e)$ and $|N[v_{\max}]|\cdot d(x) \geq e(\cH)\cdot |N[v_{\max}] \cap V(e)|$.
Thus, by Claim~\ref{claim:internal-degree-bound}, we have $e(\cH) \cdot |N[v_{\max}] \cap V(e)| \leq |N[v_{\max}]|\cdot d(x) \leq t \cdot |N[v_{\max}]|\cdot |N[v_{\max}]\cap V(e)|$. 
Noting that $N[v_{\max}]\cap V(e) \ne \varnothing$ (since $\cH$ is intersecting), this implies $e(\cH) \leq t \cdot |N[v_{\max}]|$, as desired.
\end{claimproof}

Claim~\ref{claim:edge-bound} proves the first part of Theorem~\ref{thm:char_intersect}. For the remainder of the proof, we assume $e(\cH) = t \max_{v \in V(\cH)} |N[v]|$.
By Claim~\ref{claim:edge-bound}, $t \max_{w \in V(\cH)} |N[w]| = e(\cH) \leq t \cdot |N[v_{\max}]|$ for every vertex $v_{\max}$ having maximum degree in $\cH$.  Thus,
\begin{equation} \label{eqn:maxnv} 
    \text{if $v_{\max} \in V(\cH)$ has maximum degree in $\cH$, then $|N[v_{\max}]| = \max_{w \in V(\cH)} |N[w]|$ and $e(\cH) = t \cdot |N[v_{\max}]|$.}
\end{equation}

\begin{claim}\label{claim:tight-internal-degree-bound}
For every vertex $v_{\max}$ having maximum degree in $\cH$, if $x \in N[v_{\max}]$, $e \in \cH$, and $x \notin V(e)$, then
\begin{equation*}
    d(x) = t \cdot |N[v_{\max}] \cap V(e)|.
\end{equation*}
\end{claim}
\begin{claimproof}
    Let $v_{\max}$ be a vertex having maximum degree in $\cH$.
    For every $x \in N[v_{\max}]$ and $e \in \cH$ such that $x \notin V(e)$, by Claim~\ref{claim:internal-degree-bound},  $|N[v_{\max}]|\cdot d(x) \leq t \cdot |N[v_{\max}]|\cdot |N[v_{\max}] \cap V(e)| = e(\cH)\cdot |N[v_{\max}]\cap V(e)|$, so $|N[v_{\max}]|\cdot d(x) - e(\cH)\cdot |N[v_{\max}] \cap V(e)| \leq 0$.  
    In particular, each term in the left side of~\eqref{eqn:motzkin} is at most zero, where $v_{\max}$ plays the role of $v$ in~\eqref{eqn:motzkin}, 
    so Claim~\ref{claim:motzkin} implies that each term in the left side of~\eqref{eqn:motzkin} is equal to zero, i.e., $|N[v_{\max}]| \cdot d(x) = e(\cH) \cdot |N[v_{\max}] \cap V(e)|$ for every $x \in N[v_{\max}]$ and $e \in \cH$ such that $x \notin V(e)$, so $d(x) = t \cdot |N[v_{\max}] \cap V(e)|$ (since $e(\cH) = t  |N[v_{\max}]|$) as desired.
\end{claimproof}

\begin{claim}\label{claim:samenx}
     For every $x \in V(\cH)$, we have $N[x] = V(\cH)$.
\end{claim}
\begin{claimproof}
    Let $v_{\max}$ be a vertex having maximum degree in $\cH$. 
    Let $x \in N[v_{\max}]$. We will first show that $N[x] = N[v_{\max}]$. By Claim~\ref{claim:tight-internal-degree-bound}, for every $e \in \cH$ with $v_{\max} \in V(e)$, if $x \notin V(e)$, then
    \begin{equation*}
        t \cdot |V(e)| = t \cdot |V(e) \cap N[v_{\max}]| = d(x) \leq \sum_{w \in V(e) \cap N[x]}d(x,w) \leq t \cdot |V(e) \cap N[x]| \leq t \cdot |V(e)|,
    \end{equation*}
     since $\cH$ is intersecting and has maximum codegree at most $t$.
    In particular, all inequalities above hold with equality, so $V(e) \subseteq N[x]$. If $x \in V(e)$, we also have $V(e) \subseteq N[x]$.  Thus, $V(e) \subseteq N[x]$ for every $e \in \cH$ with $v_{\max} \in V(e)$, so $N[v_{\max}] \subseteq N[x]$. 
    Hence, by~\eqref{eqn:maxnv}, we have $N[v_{\max}] = N[x]$ as desired.
    
    Next we show that $N[v_{\max}] = V(\cH)$.
    Suppose that there exists $v \in V(\cH)\setminus N[v_{\max}]$. By our  assumption, $d(v) \ge 1$, so there exists an edge $e \in \cH$ with $v \in V(e)$. Since $\cH$ is intersecting, there exists $x \in V(e) \cap N[v_{\max}]$. Since $x \in N[v_{\max}]$, we have $N[x] = N[v_{\max}]$ by the discussion in the previous paragraph. But $v \in V(e) \subseteq N[x] = N[v_{\max}]$, contradicting that $v \notin N[v_{\max}]$. Hence, $N[v_{\max}] = V(\cH)$, as desired, and since $N[x] = N[v_{\max}]$ for every $x \in N[v_{\max}]$, the claim follows.
\end{claimproof}

For all $u \ne v \in V(\cH)$, let $E(u,v) \coloneqq E_\cH(u,v)$ be the number of edges $e \in \cH$ with $u,v \in V(e)$.

\begin{claim}\label{claim:nbrhood-structure} 
    Suppose $x \in V(\cH)$ and $e \in \cH$ with $x \notin V(e)$. Then
    \begin{enumerate}[label=(\roman*)]
        \item \label{claim6i} every $w \in V(e)$ satisfies $d(x,w) = t$, and
        \item \label{claim6ii} all distinct $w_1 , w_2 \in V(e)$ satisfy $E(x,w_1) \cap E(x,w_2) = \varnothing$.
    \end{enumerate}
      In particular, $d(x) = t \cdot |V(e)|$.
\end{claim}
\begin{claimproof}
    Let $x \in V(\cH)$ and $e \in \cH$ with $x \notin V(e)$.
    By Claims~\ref{claim:tight-internal-degree-bound} and~\ref{claim:samenx}, since $\cH$ is intersecting and has maximum codegree at most $t$,
    \begin{equation*}
        t \cdot |V(e)| = d(x) \leq \sum_{w \in V(e)} d(x,w) \leq t \cdot |V(e)|.
    \end{equation*}
    In particular, all inequalities above hold with equality, and the claim follows.\COMMENT{Since the maximum codegree of $\cH$ is at most $t$, we have $d(x, w) = t$ for every $w \in V(e)$, as desired.  Moreover, no edge containing $x$ is counted more than once in the summation $\sum_{w \in V(e)}d(x, w)$ since $d(x) = \sum_{w \in V(e)}d(x, w)$, so $E(x, w_1) \cap E(x, w_2) = \varnothing$ for all distinct $w_1, w_2 \in V(e)$, as desired.} 
\end{claimproof}

\begin{claim}\label{claim:t-fold-of-linear}
The following statements hold.
\begin{enumerate}[label=(\roman*)]
    \item \label{claim7i} If $e,f \in \cH$ satisfy $|V(e) \cap V(f)| \geq 2$, then $V(e) = V(f)$.
    \item \label{claim7ii} Every $e \in \cH$ satisfies $|\{f \in \cH : V(f)=V(e) \}| = t$.
\end{enumerate} 

\end{claim}
\begin{claimproof}
To prove~\ref{claim7i}, let $e,f \in \cH$ satisfy $|V(e) \cap V(f)| \geq 2$, and let $x, y \in V(e) \cap V(f)$ be two distinct vertices.  
If $V(e) \ne V(f)$, then without loss of generality, we may assume that there exists $w \in V(e) \setminus V(f)$. 
Since $x,y \in V(f)$, we have $E(w,x) \cap E(w,y) = \varnothing$ by Claim~\ref{claim:nbrhood-structure}, 
but this is a contradiction since $e \in E(w,x) \cap E(w,y)$. Thus, $V(e) = V(f)$, as desired.

To prove~\ref{claim7ii}, let $e \in \cH$, and let $x, y \in V(e)$ be two distinct vertices.  
To show $|\{f \in \cH : V(f)=V(e) \}| = t$, by~\ref{claim7i} it suffices to show $d(x, y) = t$.
Since $d(x) < e(\cH)$, there exists $f \in \cH$ with $x \notin V(f)$. 
By Claim~\ref{claim:nbrhood-structure}, we have $d(x) =  t \cdot |V(f)| \geq 2t$ (since $\cH$ has no edge of size one).  
If all edges of $\cH$ containing $x$ also contain $y$, then $t \geq d(x,y) = d(x) \geq 2t$, a contradiction, 
so there exists an edge $g \in \cH$ with $x \in V(g)$ and $y \notin V(g)$.  By Claim~\ref{claim:nbrhood-structure}\ref{claim6i}, $d(x, y) = t$, as desired.
\end{claimproof}

By Claim~\ref{claim:t-fold-of-linear}, we deduce that $\cH$ is a $t$-fold duplication of some $n$-vertex linear intersecting hypergraph $\cH_0$. 
By~\eqref{eqn:maxnv} and Theorem~\ref{thm:charlinear}, either $\cH_0$ is an $|N[v_{\max} ]|$-vertex projective plane of order $k$ (so there exists $k\in \mathbb{N}$ with $|N[v_{\max}]| = k^2 + k + 1$), or $\cH_0$ is an $|N[v_{\max}]|$-vertex near-pencil. This completes the proof.
\end{proof}

\providecommand{\bysame}{\leavevmode\hbox to3em{\hrulefill}\thinspace}
\providecommand{\MR}{\relax\ifhmode\unskip\space\fi MR }
\providecommand{\MRhref}[2]{%
  \href{http://www.ams.org/mathscinet-getitem?mr=#1}{#2}
}
\providecommand{\href}[2]{#2}

\vspace{1cm}

{\footnotesize \obeylines \parindent=0pt
\vspace{0.3cm}
Dong Yeap Kang, Extremal Combinatorics and Probability Group, Institute for Basic Science, Daejeon, South Korea

\vspace{.3cm}
Tom Kelly, School of Mathematics, Georgia Institute of Technology, Atlanta, GA, USA

\vspace{0.3cm}
Daniela K\"{u}hn, School of Mathematics, University of Birmingham, Edgbaston, Birmingham, B15 2TT, UK

\vspace{.3cm}
Abhishek Methuku, Department of Mathematics, University of Illinois Urbana-Champaign, Urbana, IL, USA

\vspace{.3cm}
Deryk Osthus, School of Mathematics, University of Birmingham, Edgbaston, Birmingham, B15 2TT, UK
\vspace{.3cm}}

{\footnotesize \parindent=0pt
\begin{flushleft}
{\it{E-mail addresses}:}
\tt{dykang.math@ibs.re.kr, tom.kelly@gatech.edu, d.kuhn@bham.ac.uk, abhishekmethuku@gmail.com, d.osthus@bham.ac.uk}
\end{flushleft}
}

\end{document}